\documentclass[a4paper,10pt,leqno]{amsart}
\date{November 4, 2018}
\usepackage{graphicx} 
\usepackage{verbatim,enumerate}
\title[]{%
  Hypersurfaces with light-like points \\
in a Lorentzian manifold
}
\author[]{
        M.~Umehara and   % Masaaki Umehara
        K.~Yamada % Kotaro Yamada
}   

\address[Umehara]{%
   Department of Mathematical and Computing Sciences,
   Tokyo Institute of Technology,
   Tokyo 152-8552, Japan
}
\email{umehara@is.titech.ac.jp}

\address[Yamada]{%
   Department of Mathematics,
   Tokyo Institute of Technology,
   Tokyo 152-8551, Japan
}
\email{kotaro@math.titech.ac.jp}


\thanks{
Umehara was partially supported by 
the Grant-in-Aid for Scientific Research  (A) No.\ 26247005,  
and Yamada by (B) No.\  17H02839 from Japan Society for the 
Promotion of Science.
}

\usepackage{amsthm}
\theoremstyle{plain}
 \newtheorem{introtheorem}{Theorem}
 \newtheorem{introcorollary}[introtheorem]{Corollary}
 
 \newtheorem{theorem}{Theorem}[section]
 \newtheorem{proposition}[theorem]{Proposition}

 \newtheorem{lemma}[theorem]{Lemma}
 \newtheorem{corollary}[theorem]{Corollary}

\theoremstyle{definition}
 \newtheorem{definition}[theorem]{Definition}
\theoremstyle{remark}
 \newtheorem{remark}[theorem]{Remark}
 \newtheorem*{remark*}{Remark}
 \newtheorem{example}[theorem]{Example}
 \newtheorem*{acknowledgement}{Acknowledgement}

\numberwithin{equation}{section}

\newcommand{\op}[1]{{\operatorname{#1}}}

\newcommand{\mb}[1]{\vect{#1}}
\newcommand{\mc}[1]{\mathcal{#1}}

\newcommand{\vect}[1]{\boldsymbol{#1}}
\newcommand{\R}{\boldsymbol{R}}

\newcommand{\XX}{\mathcal{X}}
\newcommand{\Y}{\mathcal{Y}}
\newcommand{\Z}{\mathcal{Z}}
\renewcommand{\phi}{\varphi}

\newcommand{\pmt}[1]{{\begin{pmatrix} #1  \end{pmatrix}}}

\begin{document}
\maketitle

\begin{abstract}
Consider a constant mean curvature immersion $F:U(\subset \R^n)\to M$
into an arbitrary Lorentzian $(n+1)$-manifold $M$.
A point $o\in U$ is called
a {\it light-like point} if 
the first fundamental form $ds^2$ of $F$
degenerates at $o$.
We denote by
$B_F$ the determinant function of 
the symmetric matrix associated to $ds^2$
with respect to a local coordinate system at $o$.
A light-like point $o$ is said to be {\it degenerate}
if the exterior derivative of $B_F$ vanishes at $o$.
We show that if $o$ is a degenerate light-like point, 
then the image of $F$ contains a 
light-like geodesic segment of $M$
passing through $f(o)$ (cf. Theorem E).
This explains why several known examples of
constant mean curvature hypersurface 
in the Lorentz-Minkowski $(n+1)$-space form
$\R^{n+1}_1$ contain light-like lines
on their sets of light-like points, under a
suitable regularity condition of $F$.
Several related results are also given.  
\end{abstract}

\section*{Introduction} \label{sec:1} 
Let $M$ be a $C^r$-differentiable
oriented Lorentzian $(n+1)$-manifold
($r\ge 3$, $n\ge 2$ and
$r=\omega$ means real analyticity.)

Let $U$ be a domain of $\R^n$, and take an arbitrarily 
fixed point $o\in U$. 
We consider a $C^r$-immersion 
$F:(U,o)\to M$ ($r\ge 3$),
where
$(U;u_1,\dots,u_n)$ is a local coordinate neighborhood 
centered at $o$. 
Then the first fundamental form of $F$ can be
written as
\begin{equation}
 ds^2=\sum_{i,j=1}^n s_{i,j}^{}du_i\,du_j.
\end{equation}
We set 
\begin{equation}\label{eq:def01}
 B_F
  :=\det\pmt{
    s_{1,1}^{} &\cdots & s_{1,n}^{} \\
    \vdots &\ddots & \vdots \\
    s_{n,1}^{} &\cdots & s_{n,n}^{} 
  },
\end{equation}
and denote by $U_+$ (resp.\ $U_-$)
the space-like (resp.\ time-like) part 
of $F$, that is,
\begin{equation}\label{eq:def02}
  U_+:=\{p\in U\,;\, B_F(p)>0\},\qquad
  U_-:=\{p\in U\,;\, B_F(p)<0\}.
\end{equation}
The area element of $F$ is an $n$-form on $U$ given by
\begin{equation}\label{eq:def03}
\theta_F:=\sqrt{|B_F|}\,du_1\wedge \cdots \wedge du_{n},
\end{equation}
which does not depend on the choice of
positively oriented 
local coordinate system $(u_1,\dots,u_n)$ at $o$.
Each point on 
\begin{equation}\label{eq:Sigma}
  \Sigma_F:=U\setminus (U_+\cup U_-)
\end{equation}
is called a {\it light-like point}.
A light-like point $p\in \Sigma_F$ is called {\it degenerate}
(resp.\ {\it non-degenerate}) 
if the exterior derivative of $B_F$ vanishes
(resp. does not vanish) at $p$.
We are interested in immersions 
$F:(U,o)\to M$ whose base point $o(\in U)$ is  
light-like. 
We denote by $\vect{H}_F$ the {\it mean curvature 
vector field} of $F$ defined on $U_+\cup U_-$
(cf.\ \eqref{eq:HF}).
Although $\vect{H}_F$ diverges at $\Sigma_F$
in general, we can consider a
class of hypersurfaces
$\Y^r(M,\hat o)$ ($r\ge 4$, $\hat o:=F(o)$)
consisting of $C^r$-immersions $F$ at $o$ such that $F(o)=\hat o$,
and  $\vect{H}_F$ can be extended on a neighborhood of $o$ 
with $C^{r-2}$-differentiability
($C^\omega$-differentiability means real analyticity, and
in this case, we mean $r-2=\omega$).

On the other hand,
we can define a $C^{r-2}$-function
$
A_F:U\to \R
$
such that 
$$
H_F:=\frac{A_F}{n|B_F|^{3/2}}
$$ 
gives the mean curvature function of $F$ on $U\setminus \Sigma_F$
(see \eqref{eq:AF} for details).
An immersion $F:(U,o)\to M$ 
is called of {\it constant mean curvature} $c$
if $A_F-nc |B_F|^{3/2}$ vanishes identically.
In particular, if $c=0$, then the condition reduces to $A_F=0$,
and such hypersurfaces are called of {\it zero mean curvature}.
We denote by
$\Z^r(M,\hat o)$ the set of germs of 
zero mean curvature $C^r$-immersions $F$ at $o$ such that $F(o)=\hat o$.
Then 
$$
\Z^r(M,\hat o)\subset \Y^r(M,\hat o)
$$
holds by definition. 
Moreover, for each non-negative real number 
$\alpha$, we newly consider a wider class
$\XX^{r}_{\alpha}(M,\hat o)$ 
such that $F\in \XX^{r}_{\alpha}(M,\hat o)$ means
there exists a $C^{r-2}$-function $\phi$
such that $A_F-(B_F)^{1+\alpha}\phi$ vanishes identically
(if $\alpha\not \in \mb Z$, we mean 
$(B_F)^{1+\alpha}:=|B_F|^{1+\alpha}$).
Then it holds that (see Propositions 
\ref{prop:000} and \ref{prop:3} in Section 2):
\begin{itemize}
\item[(C1)] $\Y^r(M,\hat o)=\XX^{r}_{1}(M,\hat o)$. 
\item[(C2)] 
$F\in \XX^{r}_{0}(M,\hat o)$ if and only if
the mean curvature form $\op{sgn}(B_F)H_F\theta_F$
(cf. \eqref{eq:area}) can be  
extended as a $C^{r-2}$-form on $U$.
\item[(C3)] 
If $F\in \XX^{r}_{1/2}(M,\hat o)$, 
then $H_F$ can be extended as  a $C^{r-2}$-function 
near $o$. Constant mean curvature hypersurfaces 
belong to this class.
\item[(C4)] 
For each non-negative integer 
$\alpha$,
$\XX^{r}_{\alpha+1}(M,\hat o)$ is a subset of
$\XX^{r}_{\alpha}(M,\hat o)$.
\end{itemize}

We can show that
$\Y^r(M,\hat o)$ is non-empty (cf. Example \ref{ex:X2} and
Corollary \ref{cor:ZYX}),
and is a proper subset of
$\XX^{r}_{0}(M,\hat o)$
(cf. Theorem~\ref{thm:EA}).
We then prove the following:

\begin{introtheorem}\label{thm:main}
 Suppose that $F:U\to M$ 
is an immersion satisfying
$F\in \XX^{4}_{0}(M,\hat o)$.
If $o$ is a degenerate light-like point,
then $F(\Sigma_F)$ contains a light-like geodesic 
segment in $M$ passing through $\hat o(=F(o))$
consisting of only degenerate light-like points.
\end{introtheorem}

The proof of this result is given in Section 3.
There are infinitely many examples
belonging to $\XX^{4}_{0}(M,\hat o)\setminus 
\Y^{4}_{0}(M,\hat o)$
 (cf. Corollary \ref{cor:ZYX}).
So the theorem generalizes 
the authors' previous result \cite[Theorem 4.2]{UY}
even in the case of $M=\R^3_1$
(see also Remark \ref{new}).

A hypersurface $F$ is said to be {\it light-like}
if $B_F$ vanishes identically
(cf. Definition \ref{def:B0}).
Existence of infinitely many light-like hypersurfaces in a given
Lorentzian manifold is shown in Section 1
(cf. Theorem \ref{thm:EB}).
If $B_F$ vanishes identically, then so does $A_F$ (cf. Proposition 1.4).
In particular, light-like  hypersurfaces belong to
the class $\Z^{4}(M,\hat o)(\subset \XX^{4}(M,\hat o))$.
So we get the following corollary. 
(See  the end of Section 3. 
An alternative proof under the assumption of
$C^2$-differentiability
is given in Section 4.)	
\begin{introcorollary}
Suppose that
$F:U\to M$ is a $C^2$-differentiable light-like hypersurface.
Then $F(U)$ is foliated by light-like geodesics.
\end{introcorollary}

Let $I$ be an interval of $\R$.
A regular curve $\sigma:I\to M$ 
is called {\it null}
if 
\begin{equation}
\label{eq:null}
g(\sigma'(t),\sigma'(t))=0
\end{equation}
holds at each $t\in I$, where
$g$ is the Lorentzian metric of $M$
and $\sigma'(t):=d\sigma/dt$.
As a consequence of Theorem A, 
we can prove the following in the case $n=2$:

\begin{introcorollary}\label{cor:main}
 Suppose that $n=2$ $($i.e.\ $M$ is a Lorentzian $3$-manifold$)$
and $F:U\to M$ satisfies $F\in \XX^{4}_{0}(M,\hat o)$.
Then one of the following  two cases occurs:
 \begin{enumerate}
\item[{\rm (a)}] 
	       The point $o$ is a non-degenerate light-like point,
	       and there exists a null regular curve
$\sigma(t)$ in $M$ parametrizing
$F(\Sigma_F)$ such that the acceleration vector
$\sigma''(t)$ $($see \eqref{eq:acc} for the definition$)$
is linearly independent of the velocity vector $\sigma'(t)$.
Moreover, $F$ changes 
its causal type from space-like to time-like
across the curve $\sigma$.
\item[{\rm(b)}] 
The point $o$ is a degenerate light-like point,
and $F(\Sigma_F)$ contains 
a light-like geodesic segment
	       in $M$ passing through $\hat o(=F(o))$.
 \end{enumerate}
\end{introcorollary}

The proof of this result is given in Section 4.
Fundamental properties of space-like zero mean curvature
surfaces are discussed in \cite{ER,K,UY1}.
The article \cite{Okayama} discusses about
zero-mean curvature surfaces in $\R^{3}_1$ 
satisfying (a). Several properly embedded zero mean 
curvature surfaces of type (a) in $\R^3_1$ were 
constructed in \cite{A,maxfold,OJM}, and
examples of zero-mean curvature 
surfaces satisfying
(b) has been also constructed in \cite{A,A3,OJM,CR,HK,UY}.
(In \cite{UY}, the general existence theorem
of surface germs which changes their causal types
along light-like lines was given.
Recently, Hashimoto and Kato \cite{HK}
gave concrete examples which change their causal types
along degenerate light-like lines.)

Let $\R^{n+1}_1$ be the $(n+1)$-dimensional
Lorentz-Minkowski space.

Even when $M=\R^3_1$ and  $F$ is a zero-mean curvature
immersion, Corollary C is highly non-trivial 
(cf. see Klyachin \cite{Kl}, where the line theroem 
for  ZMC-surface in
$\R^3_1$ was proved under the $C^3$-differentiability).
We set $\mathcal Y^\omega(\R^{3}_1)
:=\mathcal Y^\omega(\R^{3}_1,\hat o)$.
A generalization of the result of
Klyachin \cite{Kl} for $F\in \Y^r(\R^{3}_1)$ $(r\ge 4)$
was given in the authors' previous work \cite{UY},
using a new approach.
In this paper, that approach is further developed to get
the above results.
As a consequence, Corollary C holds   
not only for $F$ in $\Y^{r}(M,\hat o)$,
but also for the classes $\XX^{4}_{0}(M,\hat o)$.

For surfaces $F$ that do not belong to 
$\XX^{r}_{0}(M,\hat o)$, 
the proof of Theorem A does not work. 
Real analytic hypersurfaces
with bounded mean curvature function
may not belong to $\XX^{\omega}_{0}(M,\hat o)$,
in general. However, we prove (in Section 5)
the following:

\medskip
\noindent
{\bf Theorem D.}
{\it Let $F:U\to M$
be a real analytic immersion 
with a light-like point $o\in U$
whose mean 
curvature function $H_F$ can be extended to the whole of $U$
real analytically.
If $\log |H_F|$
is bounded on $U$,
then the image of $F$ contains a light-like geodesic 
segment in $M$ passing through $F(o)$.}

\medskip
This assertion  does not 
require us to assume that
$o$ be a degenerate light-like point
because of the authors' recent joint work \cite[Theorem 1.1]{HKKUY}
with Honda, Koiso and Kokubu.
To prove  Theorem D, 
we modify the proof of Theorem A. 
Hypersurfaces with non-constant mean curvature 
satisfy the assumption of Theorem D.
So, we get the following:  

\medskip
\noindent
{\bf Theorem E.}
{\it If a $C^\omega$-differentiable
$($resp. $C^{4}$-differentiable$)$ 
non-zero constant $($resp.\ zero$)$ 
mean curvature hypersurface 
immersed in a Lorentzian $(n+1)$-manifold
$(n\ge 2)$ admits a light-like point
$($resp. a degenerate light-like point$)$, then the 
hypersurface contains a light-like geodesic
segment of the ambient manifold consisting of
the image of degenerate light-like points.}

\medskip
This theorem explains why known examples of constant 
mean curvature surfaces in the $3$-dimensional
Lorentzian space forms often contain degenerate 
light-like geodesics.
Examples with nonzero constant mean curvature
containing light-like lines in $\R^3_1$ and
the de Sitter 3-space are given in
\cite[Examples 2.7 and 2.8]{HKKUY} and \cite{FKKRUYY}, respectively.

\section{Preliminaries}
Let $M$ be a $C^r$-differentiable oriented 
Lorentzian $(n+1)$-manifold with Lorentzian metric $g$. 
Since any $C^r$-manifold
is uniquely smoothable to a $C^\omega$-structure,
we may assume $r=\omega$ without loss of generality.
We fix a point $\hat o\in M$ arbitrarily.
We denote by $\tilde {\mc I}^r(M,\hat o)$
($r\ge 2$) the set of
germs of $C^r$-immersions into $M$ 
such that $F(o)=\hat o$
(when $r=\omega$, 
$F\in \tilde {\mc I}^\omega(M,\hat o)$ means 
that $F$ is real analytic.)
We fix such an immersion
$$
F:U\to M
$$
belonging to $\tilde {\mc I}^r(M,\hat o)$, where
$U$ is a domain (i.e. connected open subset) of
$(\R^n;u_1,\dots,u_n)$
containing the origin $o$.
Then, each point $p\in U$ is 
called {\it space-like} (resp. {\it time-like, light-like}) 
if the tangent hyperplane of $F$ at $F(p)$ 
is space-like (resp. time-like, light-like)
with respect to the Lorentzian metric $g$.
We denote by $U_+$ (resp. $U_-$) 
the set of space-like (resp. time-like) points 
(cf. \eqref{eq:def02}).

\begin{definition} \label{eq:defA}
A local coordinate system $(x_0,\dots,x_n)$ ($t:=x_0$)
of $M$ centered at $\hat o$ is called {\it admissible}
if it satisfies the following two properties:
\begin{enumerate}
\item $g_{0,0}^{}=-1$,\,\, $g_{0,i}^{}=0$ 
and $g_{j,k}^{}=\delta_{j,k}$ hold
for $i,j,k=1,\dots,n$ 
at $\hat o$,
where
\begin{equation}\label{eq:g-coe}
g:=\sum_{i,j=0}^{n} g_{i,j}^{}dx_i dx_j.
\end{equation}
\item All the Christoffel symbols with respect to $g$
vanish at $\hat o$. In particular,
all derivatives $\partial g_{j,k}^{}/\partial x_i$
$(i,j,k=0,\dots,n)$ vanish at $\hat o$.
\end{enumerate}
\end{definition}

The normal coordinate system at $\hat o$ and the
 Fermi-coordinate system
along a light-like geodesic passing through $\hat o$
(see Appendix A) are admissible.
We fix an admissible coordinate system 
\begin{equation}\label{eq:coordinate}
(t,x_1,\dots,x_n)
\end{equation}
of $M$ centered at $\hat{o}$.

Let $U$ be a domain of $(\R^n;u_1,\dots,u_n)$ containing the
origin $o$.
We fix $F:(U,o)\to (M,\hat o)$ to be an immersion belonging 
to $\tilde {\mc I}^r(M,\hat o)$
such that $F(o)=\hat o$. Then the first fundamental form 
of $F$ can be written as
\begin{equation}\label{eq:ds}
  ds^2=\sum_{i,j=1}^n s_{i,j}^{}du_idu_j,\qquad
  s_{i,j}^{}:=
  g(dF(\partial/\partial u_i),dF(\partial/\partial u_j)).
\end{equation}
We set (cf.\ \eqref{eq:def01})
\begin{equation}\label{eq:BF}
B_F:=
  \det(S_F),\qquad
S_F:=\pmt{
  s_{1,1}^{} &\cdots & s_{1,n}^{} \\
 \vdots &\ddots & \vdots \\
 s_{n,1}^{} &\cdots & s_{n,n}^{}}. 
\end{equation}
We set
$$
F_{u_i}:=dF(\partial/\partial u_i)\qquad (i=1,\dots,n).
$$
Then we can write
$$
F_{u_i}= \sum_{j=0}^n \mu_{i}^{j} \partial_{x_j},
$$
where $\partial_{x_i}:=\partial/ \partial x_i$.
We then set
$$
\mu_{i,j}:=\sum_{k=0}^n \mu_{i}^{k}g_{k,j} \qquad (i=1,\dots,n,\,\, j=0,\dots n).
$$
For $(\xi_0,\xi_1,\dots,\xi_n)\in \R^{n+1}$, we can write 
$$
\det
\pmt{
\mu_{1,0} & \mu_{1,1} &\cdots & \mu_{1,n} \\
\mu_{2,0} & \mu_{2,1} &\cdots &\mu_{2,n} \\
\vdots    & \vdots & \ddots &  \vdots \\
\mu_{n,0} & \mu_{n,1} &\cdots & \mu_{n,n} \\
\xi_0 &   \xi_1 & \dots & \xi_n}
=\sum_{i=0}^n \xi_i \tilde\nu_i,
$$
where $\tilde \nu_{i}:U\to \R$ is a polynomial of degree $n$ 
in $(\mu_{j,k})_{j=1,...n,\,\,k=0,...,n}$.
Then
\begin{equation}\label{eq:nu0}
\tilde \nu:=\sum_{i=0}^n \tilde \nu_i \partial_{x_i}
\end{equation}
gives a normal vector field of $F$.
Moreover, 
$$
g(\tilde \nu,\tilde \nu)=-B_F
$$
holds.
We denote by $D$ the Levi-Civita connection associated to 
$g$,
and set
\begin{equation}\label{eq:AF}
 A_F:=\sum_{i,j=1}^n \tilde s^{i,j}\tilde h_{i,j},\qquad
  \tilde h_{i,j}
  :=g(D_{\partial /\partial u_i}
  dF(\partial /\partial u_j),\tilde \nu),
\end{equation}
where $(\tilde s^{i,j})_{i,j=1,\dots,n}$
is the cofactor matrix of 
$S_F$ as in \eqref{eq:BF}.
The mean curvature function $H_F$ of $F$ 
(with respect to the unit normal vector field
$\tilde \nu/\sqrt{|g(\tilde \nu,\tilde \nu)|}$)
and the mean curvature vector field $\vect{H}_F$ are
defined on the set of space-like or time-like points
by
\begin{equation}\label{eq:HF}
 H_F:=\frac{A_F}{n|B_F|^{3/2}},\qquad 
\vect{H}_F:=\frac{A_F}{n(B_F)^2}\tilde\nu,
\end{equation}
respectively.
We defined the area element
$\theta_F$ on $U$ as in the introduction
(cf. \eqref{eq:def03}).
Then 
\begin{equation}\label{eq:area}
\omega_H:=\frac{A_F}{nB_F}
du_1\wedge \cdots \wedge du_n=\op{sgn}(B_F)H_F \theta_F
\end{equation}
is defined on $U\setminus \Sigma_F$,
and is called the {\it mean curvature form} of $F$,
where $\op{sgn}(B_F)$ is the sign of the
function $B_F$.

\begin{remark}\label{rmk:1-2}
Since $H_F$ is not well-defined 
at the light-like points, there is another
possibility of the definitions of
mean curvature function and
mean  curvature field on $U\setminus \Sigma$
as follows:
\begin{equation}\label{eq:HF2}
 \hat H_F:=\op{sgn}(B_F)\frac{A_F}{n|B_F|^{3/2}},\qquad 
\hat {\vect{H}}_F:=\psi_F
\tilde\nu
\quad \left(\psi_F:=\op{sgn}(B_F)\frac{A_F}{n(B_F)^2}\right),
\end{equation}
respectively.
Then the mean curvature form
satisfies $\omega_H:=\hat H_F \theta_F$.
However, $\hat{\vect{H}}_F$ 
cannot be real analytic  when 
$U_+,U_-$ are not empty,
because of the factor $\op{sgn}(B_F)$.
On the other hand,
there is an example such that $\hat H_F$
is real analytic but $H_F$ is not
(cf. Example \ref{ex:X1}). However, 
the replacement of $H_F$ by $\hat H_F$
does not affect the statement of Theorem D, 
since hypersurfaces satisfying the condition of 
Theorem D never change their causal types (cf. 
\cite[Theorem 1.1]{HKKUY}).
\end{remark}

When $M=\R^{n+1}_1$, the 
explicit expressions of $A_F$ and $B_F$
are given in Appendix~B.

We fix $F:(U,o)\to (M,\hat o)$ to be an 
immersion belonging to $\tilde {\mc I}^r(M,\hat o)$,
where $r\ge 2$.
We denote by $\tilde {\mc I}^r_L(M,\hat o)
(\subset \tilde {\mc I}^r(M,\hat o))$
the set of germs of $C^r$-immersions such that $o$
is a light-like point.
If $F$ is in the class $\tilde {\mc I}^r_L(M,\hat o)$,
then the tangent hyperplane of the image of $F$ at $o$ contains 
a light-like vector, but does not contain any time-like vectors.
Thus, the image of $F$ can be  expressed as a graph of a function
$f$ defined on a certain neighborhood of the origin $o$,
that is, we can write
\begin{equation}\label{eq:Ftof}
   F(x_1,\dots,x_n)=(f(x_1,\dots,x_n),x_1,\dots,x_n),
\end{equation}
where $(x_0,x_1,\dots,x_n)$ is an fixed admissible coordinate
system centered at $\hat o:=F(o)$.
We call $f$ the {\it height function} induced by $F$, and denote it by
 $\iota_F:=f$.

If $M$ is the Lorentz-Minkowski space $\R^{n+1}_1$
and \eqref{eq:coordinate} is the canonical coordinate system,
it holds at $o$ that (see Appendix B for details)
\begin{align}
\label{eq:BF1}
B_F&=1-(f_{x_1})^2-\cdots -(f_{x_n})^2,\\
\label{eq:AF1}
A_F&=B_F \triangle f-
\frac12 \nabla B_F\star \nabla f,
\end{align}
where 
the star-dot \lq$\star$\rq\ denotes the canonical Euclidean 
inner product of $\R^{n}$, and
\begin{align*}
&f_{x_i}:=\frac{\partial f}{\partial x_i},\quad
f_{x_i,x_j}:=\frac{\partial^2 f}{\partial x_i\partial x_j}
\qquad (i,j=1,\dots,n),\\
&\nabla:=\left(\frac{\partial}{\partial x_1},
\dots,\frac{\partial}{\partial x_n}\right),\qquad
\triangle f:=\sum_{i=1}^n f_{x_i,x_i}.
\end{align*}

\medskip
We now return to the case of general $M$.  
By a suitable rotation 
of the coordinate system
with respect to the $t$-axis
($t:=x_0$),
we may assume
\begin{equation}\label{eq:c}
f(o)=0,\quad
f_{x_1}(o)=\cdots=f_{x_{n-1}}(o)=0,\quad
f_{x_n}(o)=1.
\end{equation}
We set 
$
\partial_{x_k}:=dF(\partial/\partial x_k)
$ ($k=1,\dots,n$).
Since we are using an admissible coordinate system,
the computation of 
the value of
$s_{i,j}^{}$ (cf. \eqref{eq:ds})
at $o$
is completely the same
as in the case of $\R^{n+1}_1$, so
we have that (cf. \eqref{eq:b1})
\begin{equation}
s_{i,j}^{}(o)=\delta_{i,j}-f_{x_i}(o)f_{x_j}(o)\qquad 
(i,j=1,\dots,n),
\end{equation}
where $\delta_{i,j}$ is Kronecker's delta.
In particular, \eqref{eq:BF1} holds at $o$.
Since all of the Christoffel symbols of the metric $g$
vanish at $\hat o(=F(o))$, 
\eqref{eq:AF1} holds at $\hat o$.

We denote by $\mc I^r_L(M,\hat o)$
the set of germs of $C^r$-immersions 
$F\in \tilde{\mc I}^r_L(M,\hat o)$
satisfying \eqref{eq:c}.

\begin{definition}\label{def:C1C2}
We denote by $C^r_0(\R^n)$ the set of 
$C^r$-function germs at $o\in \R^n$.
Let $C^r_1(\R^n)$ 
be the set of $C^r$-function 
germs $f\in C^r_0(\R^n)$ 
at $o\in \R^n$ satisfying $f(o)=0$. 
Moreover, $C^r_2(\R^{n})$ is the set of $C^r$-functions 
$\phi\in C^r_1(\R^{n})$
satisfying
$
\nabla\phi(0,\dots,0)=\vect{0}.
$
\end{definition}
By \eqref{eq:Ftof}, the map 
\[
  \iota:\mc I^r_L(M,\hat o)\ni F\to
 \iota_F\in \bigl\{f\in C^r_1(\R^n)\,;\,
  \mbox{$f$ satisfies \eqref{eq:c}}\bigr\}
\]
is induced.
The following assertion holds:
\begin{proposition}\label{prop:B0}
 If $p\in \Sigma_F$ is a degenerate light-like point,
 then $A_F$ vanishes at $p$.
 In particular, if $B_F$ vanishes identically, 
 then so does $A_F$.
\end{proposition}

\begin{proof}
If we take an admissible coordinate system
at $p$, then all of the Christoffel symbols vanish
at $p$.  So it is
sufficient to show the case 
$M=\R^{n+1}_1$.
Since $p$ is a degenerate light-like point of $F$,
$B_F$ and $\nabla B_F$ vanish at $p$.
Thus, \eqref{eq:AF1} yields that $A_F(p)=0$. 
\end{proof}

\begin{definition}\label{def:B0}
We denote by 
$
\Lambda^r(M,\hat o)
$ the set of germs
of immersions $F\in \mathcal I^r_L(M,\hat o)$ with
identically vanishing $B_F$.
In other words, $\Lambda^r(M,\hat o)$
can be considered as the set of germs 
of light-like hypersurfaces. 
\end{definition}

A typical example of a light-like hypersurface in $\R^{n+1}_1$
is the light-cone with vertex at $(-1,0,0,-1)$,
 which is the graph of the function
\begin{equation}\label{eq:L-cone}
f(x_1,\dots,x_{n}):=\sqrt{(x_1)^2+\cdots+(x_{n}+1)^2}-1.
\end{equation}

The following assertion is a general existence theorem for
light-like hypersurfaces in $M$, that is, a
generalization of  \cite[Proposition 2.2]{UY} in the case 
$n=2$ and $M=\R^3_1$.

\begin{theorem}\label{thm:EB}
Suppose that $M$ is a $C^\omega$-Lorentzian manifold. 
Then the map
$$
\lambda:\Lambda^\omega
(M,\hat o)
\ni F\mapsto (\lambda_F:=)f(x_1,\dots,x_{n-1},0)
\in C^\omega_2(\R^{n-1})
$$
is bijective, where $f:=\iota_F$.
\end{theorem}

\begin{proof}
We first consider the case that $M=\R^{n+1}_1$.
Since $f_{x_n}(o)=1$, 
the identity $B_F=0$ is equivalent to the relation
$$
f_{x_n}=\sqrt{1-\sum_{i=1}^{n-1}(f_{x_i})^2},
$$
which  gives a normal form of a partial differential equation
under the initial condition 
\begin{equation}\label{eq:iniB}
f(x_1,\dots,x_{n-1},0)=\lambda(x_1,\dots,x_{n-1})
\end{equation}
for a given $\lambda\in C^\omega_2(\R^{n-1})$.
(By \eqref{eq:c}, $f(x_1,\dots,x_{n-1},0)$
belongs to the class $C^\omega_2(\R^{n-1})$.)
So we can apply the Cauchy-Kovalevski theorem
(cf.\ \cite{T}), and get the uniqueness and 
existence of such a function $f$.

We next consider the case for general $M$
as follows: 
The function  $B_F$ can be written in terms of
$g_{i,j}^{}$ and $f_{x_k}$ ($i,j,k=1,\dots,n$).
In fact, 
$$
B_F=\op{det}(S),\qquad
S:=\biggl(f_{x_i}f_{x_j}\,g_{0,0}^{}+
 f_{x_i}g^{}_{0,j} + f_{x_j}g^{}_{i,0}+
 g_{i,j}^{}\biggr)_{i,j=1,\dots,n}
$$
holds on a neighborhood of $o$. 
Regarding $f_{x_i}$ ($i=1,\dots,n$) as variables,
we set
\begin{align*}
&\beta(t,x_1,\dots,x_n,f_1,\dots,f_{n})\\
&\phantom{}:=\op{det}
\biggl(\bigl(
f_{i}f_{j}\,g_{0,0}^{}(t,x_1,\dots,x_n)
+ f_{i} g_{0,j}^{}(t,x_1,\dots,x_n) \\
&\phantom{aaaaaaaaaaaaaaaaa}
+ f_{j} g_{i,0}^{}(t,x_1,\dots,x_n)
+g_{i,j}^{}(t,x_1,\dots,x_n)
\bigr)_{i,j=1,\dots,n}\biggr)
\end{align*}
to be a function of $t,x_1,\dots,x_n,f_1,\dots,f_{n}$ 
so that
$$
B_F=\beta(t, x_1,\dots,x_n,f_{x_1},\dots,f_{x_n}).
$$

For example, if $M=\R^{n+1}_1$,
then $\beta$ satisfies (cf. Appendix B)
$$
\beta=1-\sum_{i=1}^n (f_i)^2.
$$
 
Consider the case of general $M$.
Since $(t,x_1,\dots,x_n)$ is an admissible coordinate system, 
the coefficients $g_{i,j}^{}$ of the Lorentzian metric
given by \eqref{eq:g-coe}
satisfy ($\delta_{i,j}$ is Kronecker's delta)
\begin{equation}\label{normal1}
g_{i,j}^{}(o)=
\begin{cases}
\delta_{i,j} & \mbox{$i,j=1,\dots,n$}, \\
0           & \mbox{$i=1,\dots,n$, \,\, $j=0$}, \\
0           & \mbox{$i=0$,\,\, $j=1,\dots,n$}, \\
-1           & \mbox{$i=j=0$}, 
\end{cases} 
\end{equation}
and
\begin{equation}\label{normal2}
(g_{i,j}^{})_k(o)=0
\qquad (i,j,k=0,1,\dots,n),
\end{equation}
where $(g_{i,j}^{})_k:=\partial g_{i,j}^{}/\partial x_k$.
By considering \eqref{normal1},
the fact that the coefficient of 
$(f_{x_n})^2$ in \eqref{eq:BF} is $-1$ at $o$
implies 
$$
\frac{\partial \beta}{\partial f_{n}}=-2, 
$$
when
$t=x_1=\cdots=x_n=f_1=\cdots=f_{n-1}=0$ and  
$f_{n}=1$.
By the implicit function theorem, 
there exists an analytic function of 
several variables
$$
X:=X(t,x_1,\dots,x_n,f_1,\dots,f_{n-1})
$$
defined on a small neighborhood of the origin
of $\R^{2n-1}$ such that
$$
\beta(t,x_1,\dots,x_n,f_1,\dots,f_{n-1},X(t, x_1,\dots,x_n,f_1,\dots,f_{n-1}))=0.
$$
(For example, if $M=\R^{n+1}_1$, 
we have
$
X=(1-\sum_{k=1}^{n-1} (f_k)^2)^{1/2}
$.)
To find the desired $f$, it is sufficient to
solve the partial differential equation 
\[
  f_{x_n}=X(f,x_1,\dots,x_n,f_{x_1},\dots,f_{x_{n-1}})
\]
under the initial condition \eqref{eq:iniB}.
By the Cauchy-Kovalevski theorem (cf. \cite{T}),
the uniqueness and existence of $f$ can be shown.
\end{proof}

\begin{example}\label{ex:P}
 The light-like hyperplane
 $F(x_1,\dots,x_n):=(x_n,x_1,\dots,x_n)$
 belongs to the class $\Lambda^\omega(\R^{n+1}_1)$
 such that
 $\lambda_F(x_1,\dots,x_{n-1})=0$.
\end{example}

\begin{example}\label{ex:L}
The light-cone given in
\eqref{eq:L-cone}
 belongs to the class $\Lambda^\omega(\R^{n+1}_1)$
 such that
 $\lambda_F=\sqrt{1+(x_1)^2+\cdots+(x_{n-1})^2}-1$.
\end{example}

\section{General existence of hypersurfaces in 
$\XX^{\omega}_{\alpha}(M,\hat o)$}

Let $F:U\to M$ be an immersion
in the class $\mc I^r_L(M,\hat o)$ ($r\ge 2$) 
such that $U\setminus \Sigma_F$ is an open dense subset of $U$,
and fix a $C^{r-2}$-function  $\phi$ on $U$.
For a non-negative integer $\alpha$,
we say that 
$F$ is \emph{$(\phi,\alpha)$-admissible} 
if
\begin{equation}\label{eq:H2}
   A_F-\phi (B_F)^{\alpha+1}=0,
\end{equation} 
where $A_F$ and $B_F$ are given 
in \eqref{eq:AF} and \eqref{eq:BF},
respectively. 
On the other hand, suppose that $\alpha(>0)$ is
not an integer.
Then $F$ is said to be \emph{$(\phi,\alpha)$-admissible} 
if
\begin{equation}\label{eq:H2add}
   A_F-\phi |B_F|^{\alpha+1}=0.
\end{equation}%
We set
\begin{equation}\label{eq:Y}
\XX^{r}_{\alpha,\varphi}(M,\hat o) :=\left\{F\in\mc 
I^r_L(M,\hat o)
\,;\,\text{$F$ is $(\varphi,\alpha)$-admissible}\right\}.
\end{equation}
An immersion germ $F\in \mc I^{r}_L(M,\hat o)$ is 
called $\alpha$-\emph{admissible} if
it is $(\phi,\alpha)$-admissible for a 
certain $\phi\in C^{r-2}_0(\R^n)$.
The set
\[
  \XX^{r}_{\alpha}(M,\hat o):=\bigcup_{\phi\in C^{r-2}_0(\R^n)}
  \XX^{r}_{\alpha,\phi}(M,\hat o)
\]
consists of all germs of $\alpha$-admissible immersions.

\begin{proposition}\label{prop:000}
If $\alpha$ is a positive integer, then 
$
  \XX^{r}_{\alpha}(M,\hat o)
\subset  \XX^{r}_{\alpha-1,\phi}(M,\hat o)
$
holds, that is, {\rm (C4)} in the introduction holds. 
\end{proposition}

\begin{proof}
If we set $\psi:=\phi B_F$,
then we have
$\XX^{r}_{\alpha,\phi}(M,\hat o)
\subset \XX^{r}_{\alpha-1,\psi}(M,\hat o)$.
\end{proof}

Recall that
$
\Y^r(M,\hat o):=\XX^{r}_{1}(M,\hat o)
$
is the set of germs of 
$C^r$-immersions $F$ at $o$ such that $F(o)=\hat o$,
and  the mean curvature vector field 
$\vect{H}_F$ can be extended on a neighborhood of $o$ 
with $C^{r-2}$-differentiability.
We prove the following:

\begin{proposition}\label{prop:3}
The assertions {\rm (C1),\, (C2),\, (C3)} 
in the introduction hold.
\end{proposition}

\begin{proof}
If $\vect{H}_F$ (resp.\ $\omega_H$)
can be extended on a neighborhood of $o$ 
with $C^{r-2}$-differentiability,
then we have \eqref{eq:H2} by setting $\alpha=1$
(resp. $\alpha=0$), and we consider the function $\phi$ 
satisfying (cf.\ \eqref{eq:HF} and \eqref{eq:area})
$$
\phi=n \langle \vect{H}_F,\tilde \nu\rangle /
\sqrt{\langle\tilde \nu,\tilde \nu\rangle}
\qquad \left(\mbox{resp.\,\,\,}
\omega_H=\frac{\phi}n du_1\wedge \cdots \wedge du_n\right),
$$
where $\langle\,,\,\rangle$ denotes a certain Riemannian metric
on $M$. 
So we obtain (C1) and (C2).
The assertion (C3) immediately follows
from \eqref{eq:HF} and
\eqref{eq:H2add}.
\end{proof}

Since $\vect{H}_F$ (cf. \eqref{eq:HF})
vanishes when $\phi=0$, 
the subset   
\begin{equation}\label{eq:Z}
   \Z^r(M,\hat o)
:=\{F\in \mc I^{r}_L(M,\hat o)\,;\, 
A_F=0\}\left(\subset \Y^r(M,\hat o)\right)
\end{equation}
consists of germs of zero 
mean curvature immersions
in $M$ at the light-like point $o$.
By definition, we have (cf.\ Proposition~\ref{prop:B})
\begin{equation}
 \Lambda^r(M,\hat o)\subset 
\Z^r(M,\hat o) \subset \Y^r(M,\hat o)
\subset \XX^{r}_{0}(M,\hat o)
\qquad (r\ge 2).
\end{equation}
When $M=\R^{n+1}_1$, we denote
$\Lambda^r(M,\hat o)$, $\Z^r(M,\hat o)$, 
$\Y^r(M,\hat o)$ 
and $\XX^{r}_{\alpha}(M,\hat o)$
by
$$
\Lambda^r(\R^{n+1}_1),\quad
\Z^r(\R^{n+1}_1),\quad
\Y^r(\R^{n+1}_1),\quad \XX^{r}_{\alpha}(\R^{n+1}_1),
$$
since $\R^{n+1}_1$ is a homogeneous space.

In this paper, 
we shall generalize the authors'
previous results in \cite{UY} on
$\Y^r(\R^{3}_1)$ to the corresponding results on 
the wider class $\XX^{r}_{0}(M,\hat o)$.
We prepare the following:

\begin{lemma}\label{lem:nablaB}
Let $F$ be an immersion
belonging to the class $\XX^{3}_{\alpha}(M,\hat o)$
$(\alpha=0,1,2,\dots)$.
Then, $f:=\iota_F$ satisfies
$(B_F)_{x_n}(o)=f_{x_n,x_n}(o)=0$.
Moreover, $o$ is a degenerate light-like point
if and only if 
$
 f_{x_n,x_1}(o)=\cdots=f_{x_n,x_{n-1}}(o)=0.
$
\end{lemma}

\begin{proof}
Since $(t,x_1,\dots,x_n)$ is an admissible  coordinate system of
$M$ at $\hat o$, we can use 
\eqref{eq:BF1} and
\eqref{eq:c}.
Since $F\in \XX^{r}_{\alpha}(M,\hat o)$ ($r\ge 3$) 
and $B_F(o)=0$, 
we have (the star-dot \lq$\star$\rq\ is the 
canonical Euclidean inner product of $\R^n$)
\begin{align*}
  0&=A_F(o)-\phi(o) (B_F)^{1+\alpha}(o)=A_F(o)\\
&=-\frac12\nabla B_F(o) \star \nabla f(o)
   =-\frac12(B_F)_{x_n}(o)f_{x_n}(o)
   \left(=-f_{x_n,x_n}(o)\right),
\end{align*}
which implies the
first assertion $(B_F)_{x_n}(o)=f_{x_n,x_n}(o)=0$.
On the other hand, we have 
\begin{align}
 \label{eq:dxy0}
 (B_F)_{x_j}(o)
 &=-2\left(\sum_{i=1}^{n-1}f_{x_i}(o)f_{x_i,x_j}(o)\right)
 -2 f_{x_n,x_j}(o)f_{x_n}(o)\\
 &=-2 f_{x_n,x_j}(o) \nonumber
 \quad (j=1,\dots,n-1).
\end{align}
Then by  \eqref{eq:dxy0}, it is obvious that 
$\nabla B(o)=\mb 0$  if and only if 
$(f_{x_n,x_1},\dots,f_{x_n,x_{n-1}})$ vanishes at $o$.
\end{proof}

In this section, we  show 
the following general existence result.

\begin{theorem}\label{thm:EA}
We fix a non-negative integer $\alpha$.
For each $\phi\in C^\omega_0(\R^n)$,
the map
$$
\eta^\alpha:\XX^{\omega}_{\alpha,\phi}(M,\hat o)\ni F
\mapsto \eta^\alpha(F)\in 
C^\omega_2(\R^{n-1})\times 
C^\omega_1(\R^{n-1})
$$
defined by
$($for the definitions of 
$C_{1}^\omega(\R^{n-1})$ and $C_{2}^\omega(\R^{n-1})$,
see Definition \ref{def:C1C2}$)$
$$
{\eta^\alpha(F)}\left(=(\eta_0(F),\eta_1(F))\right)
:=\biggl(f(x_1,\dots,x_{n-1},0),f_{x_n}(x_1,\dots,x_{n-1},0))-1\biggr)
$$
is bijective. 
Moreover, the base point 
$o$ is a non-degenerate light-like point of $F$  
if and only if 
\[
 \nabla\psi:=
  (\psi_{x_1},\dots,
   \psi_{x_{n-1}})\ne \mb 0
\]
at the origin of $\R^{n-1}$,
where $\psi:=f_{x_n}(x_1,\dots,x_{n-1},0)$.
\end{theorem}

\begin{proof}
 Suppose that $F\in \XX^{\omega}_{\alpha,\phi}(M,\hat o)$. 
Firstly, we consider the case $M=\R^{n+1}_1$.
Since 
$A_F-\varphi B_F^{\alpha+1}$ vanishes identically,
 $f(=\iota_F)$ satisfies
(cf. \eqref{eq:BF1} and \eqref{eq:AF1})
 \begin{align*}
f_y&=g,\\
g_{y}&=\frac{-1}{1-Q_f}
\biggl(
(1-Q_f-g^2)^{1+\alpha}\phi-\sum_{k=1}^{n-1}
(1-Q_f+(f_{x_k})^2-g^2)f_{x_k,x_k} \\
&\phantom{aaaaaaaaaaaaaaaaaaaa}
-2\sum_{1\le j<k\le n-1}f_{x_j}f_{x_k}f_{x_j,x_k}-2 
g^{}\sum_{i=1}^n f_{x_i}g_{x_i}^{}\biggr ),
\end{align*}
where $Q_f:=\sum_{k=1}^{n-1} (f_{x_k})^2$.
(When $n=2$, $\phi=0$ and $M=\R^3_1$,
this system reduces to the case of $c=0$
as in \eqref{eq:Hc} in Section 5.)
This gives a normal form of a
system of partial differential equations. 
So we can apply the Cauchy-Kovalevski theorem (cf.\ \cite{T})
and get a solution $f$ satisfying the initial conditions
\begin{align}\label{eq:IA1}
f(x_1,\dots, x_{n-1},0)&=\eta_0(x_1,\dots, x_{n-1},0), \\
\label{eq:IA2}
f_{x_n}(x_1,\dots, x_{n-1},0)&=1+\eta_1(x_1,\dots, x_{n-1},0), 
\end{align}
where $(\eta_0,\eta_1)\in C_{2}^\omega(\R^{n-1})
\times C_{1}^\omega(\R^{n-1})$.
Then there exists a unique solution
$(f,g)$ of this system of partial differential 
equations satisfying the
initial conditions
\eqref{eq:IA1} and
\eqref{eq:IA2}. Obviously,  
$$
F:=(f(x_1,\dots,x_{n-1}),x_1,\dots,x_{n-1})
$$ 
gives the desired immersion in 
$\XX^\omega_{\alpha,\phi}(\R^{n+1}_1)$
satisfying $\eta^\alpha(F)=(\eta_0,\eta_1)$.

We next consider the case for general $M$.
We express $F$ as in
\eqref{eq:Ftof} under an admissible coordinate system 
$(t,x_1,\dots,x_n)$ satisfying \eqref{eq:c} at $\hat o$.
We set 
\[
  f_{i}:=f_{x_i},\quad 
   f_{i,j}:=D_{\partial_{x_j}}df(\partial_{x_i})
   \qquad (i,j=1,\dots,n),
\]
where $D$ is the Levi-Civita connection of $M$.
Moreover, we set
\begin{align}
\begin{aligned}\label{eq:multi}
\mb x&:=(x_1,\dots,x_n),\\
\mb f_1&:=(f_{1},f_{1,1},\dots,f_{1,n-1}),\\
\mb f_2&:=(f_{2},f_{2,2},\dots,f_{2,n-1}),\\
&\vdots \\
\mb f_{n-1}&:=(f_{n-1},\dots,f_{n-1,n-1}),\\
\mb f_n&:=(f_n,f_{1,n},\dots, f_{n,n-1}).
\end{aligned}
\end{align}
Then $A_F-(B_F)^{1+\alpha}\phi $ can be written as a function
of the variables $\mb x,f, \mb f_1,\dots,\mb f_n$,
and there exists a function $\tau$
of several variables such that
$$
A_F-(B_F)^{1+\alpha}\phi =\tau(\mb x ,f,\mb f_1,\dots,\mb f_{n},f_{n,n}).
$$
For example, if $M=\R^{n+1}_1$,
\begin{equation}\label{eq:alpha}
\tau=\left(1-\sum_{k=1}^n (f_k)^2\right)\sum_{i=1}^n f_{i,i}
+\sum_{i,j=1}^{n} f_i f_j f_{i,j}
-\left(1-\sum_{k=1}^n (f_k)^2\right)^{1+\alpha}\phi
\end{equation}
holds (cf. \eqref{eq:AF1}). 
Here, 
the coefficient of 
$f_{x_n,x_n}$ on the right-hand side of \eqref{eq:alpha}
is $-1$ at $o$.
Since $(t,x_1,\dots,x_n)$ is an admissible 
coordinate system of $M$ at $o$, we have
$$
\frac{\partial \tau}{\partial f_{n,n}}=1, 
$$
when
$t=x_i=f_i=f_{i,n}=f_{j,n}=f_{i,j}=0\,\, (i,j=1,\dots,n-1)$
and $f_{n,n}=1$.
So by the implicit function theorem,
there exists a $C^\omega$-function
$$
Y=Y(t,\mb x ,f,\mb f_1,\dots,\mb f_{n})
$$
defined on a small neighborhood of the origin
such that
$$
\tau\left(t,\mb x ,f,\mb f_1,\dots,\mb f_{n},
Y(t,\mb x ,f,\mb f_1,\dots,\mb f_{n})\right)=0.
$$
To find the desired $f$, we consider 
the following system of partial differential 
equations
\begin{align}\label{eq:pde0}
f_{x_n}&=g,\\
\nonumber
g_{x_n}^{}&=Y(t,\mb x ,f,\mb f_1,\dots,\mb f_{n-1},\mb g'),
\end{align}
where $\mb g':=(g,g_{x_1}^{},\dots,g_{x_{n-1}}^{})$.
We then apply the Cauchy-Kovalevski theorem
to this system of partial differential equations,
and get the unique solution $(f,g)$ of
\eqref{eq:pde0} satisfying the initial conditions 
\eqref{eq:IA1} and \eqref{eq:IA2}.
Obviously,  
$$
F:=(f(x_1,\dots,x_{n-1}),x_1,\dots,x_{n-1})
$$ 
is the desired immersion in 
$\XX^{\omega}_{\alpha,\phi}(M,\hat o)$.
The last assertion follows immediately from
Lemma \ref{lem:nablaB}.
\end{proof}

When $\alpha=1$ and $\phi=0$, 
we get the following:

\begin{corollary}\label{cor:Z00}
 The map
$$   \zeta:\Z^\omega(M,\hat o)\ni F \mapsto 
(\zeta_F:=)\biggl(f({\mb x_0},0),f_{x_n}({\mb x_0},0)-1\biggr)
    \in C^\omega_{2}(\R^{n-1})\times C^\omega_{1}(\R^{n-1})
$$ 
is bijective, where $f:=\iota_F$ and
${\mb x_0}=(x_1,\dots,x_{n-1})$.
Moreover, the base point 
$o$ is a non-degenerate light-like point of $F$  
if and only if 
\[
 \nabla\psi:=
  (\psi_{x_1},\dots,
   \psi_{x_{n-1}})\ne \mb 0
\]
at the origin of $\R^{n-1}$,
where $\psi:=f_{x_n}(x_1,\dots,x_{n-1},0)$.
\end{corollary}

More generally, we prove the following:

\begin{corollary}\label{cor:ZYX}
There are infinitely many surfaces $F$
having degenerate light-like points
belonging to $\Y^\omega(M,\hat o)
=\XX^\omega_1(M,\hat o)$
$($resp. $\XX^\omega_0(M,\hat o))$ but
not contained in $\Z^\omega(M,\hat o)$
$($resp. $\Y^\omega(M,\hat o))$.
\end{corollary}

\begin{proof}
We take a function germ $\phi\in C_0^\omega(\R^n)$
satisfying $\phi(o)\ne 0$, where $o$ is the origin of
$\R^n$. Also, we let $\psi\in C^\omega_2(\R^{n-1}_1)$
(cf. Definition \ref{def:C1C2}).
Then $F_1:=(\eta^1)^{-1}(\psi)$ (resp. $F_0:=(\eta^0)^{-1}(\psi)$)
does not belong to $\Z^\omega(M,\hat o)$ 
(resp. $\XX^\omega_{1,\phi}(M,\hat o)$),
since $\Z^\omega(M,\hat o)=\XX^\omega_{1,\phi}(M,\hat o)$
with $\phi=0$ (resp. since 
$F\in \XX^\omega_{1,\phi}(M,\hat o)$
implies $F\in \XX^\omega_{0,\phi B_F}(M,\hat o)$).
\end{proof}

The following is a direct consequence of
the injectivity of $\eta^{\alpha}$
and Theorem \ref{thm:EB}.

\begin{corollary}
In the above correspondence, it holds that
 \[
   \Lambda^\omega(M,\hat o)
=\left\{F\in \Z^\omega(M,\hat o)\,;\,  
    \
\zeta_F=\left(\lambda,\sqrt{1-|\nabla \lambda|^2}-1\right),
     \,\, \lambda\in C^\omega_{1}(\R^{n-1})
    \right\},
 \]
where
$
|\nabla \lambda|^2:=(\lambda_{x_1})^2
+\cdots+(\lambda_{x_{n-1}})^2.
$
\end{corollary}

\begin{example}
The graph of the entire function
$$
f_K(x_1,\dots,x_n) := (x_1+1) \tanh x_n
$$
on $\R^n$ given in Kobayashi \cite{K} 
gives a zero mean curvature 
hypersurface in $\R^{n+1}_1$ which 
changes type from space-like to time-like.
The points where $f_K$ changes type are non-degenerate
light-like points.
In this case,
$$
\zeta_F(x_1,\dots,x_{n-1})=(0,x_1),
$$
where $f_K=\iota_F$.
In particular, $F$ admits only non-degenerate light-like points. 
\end{example}

\begin{example}
The light-like plane given in Example \ref{ex:P}
and the light-cone given in Example \ref{ex:L}
have only degenerate light-like points. 
\end{example}

\section{Proof of Theorem \ref{thm:main}.}

Let $F:(U,o)\to (M,\hat o)$ be a $C^r$-immersion
belonging to $\mc I^r_L(M,\hat o)$ for $r\ge 4$.
We suppose that $o$ is a degenerate light-like point
and $F$ is written as in 
\eqref{eq:Ftof} and \eqref{eq:c}.
Then 
$$
\mb v:=dF(\partial_{x_n})\in T_{\hat o}M
$$
is a null vector, 
and there exists a light-like geodesic 
$\sigma:[-\sqrt{2}\epsilon,
\sqrt{2}\epsilon]
\to M$ for a sufficiently small $\epsilon(>0)$
such that $\sigma(0)=\hat o$ and $\sigma'(0)=\mb v$.
By Proposition A.1 in Appendix A, 
we can take the Fermi coordinate system
$(x_0,x_1,\dots,x_n)$ centered at $o$ defined on the
tubular neighborhood $U_\sigma$ of $\sigma$:
$$
|x_0|<\epsilon,\,\, |x_n|<\epsilon,
\quad |x_i|<\epsilon \qquad (i=1,\dots,n-1),
$$
where $\epsilon>0$ is a sufficiently small number.
This gives an admissible coordinate system (cf. Definition~\ref{eq:defA})
centered at $\hat o$ such that
$$
\sigma(t):=(t,0,\dots,0,t)\qquad (|t|<\epsilon).
$$
Moreover, 
the properties (a2) and (a3) in 
Proposition A.1 in Appendix A are satisfied.

The strategy of the proof 
of Theorem \ref{thm:main}
is essentially same as that of the case of $n=2$ and
$M=\R^3_1$ given in \cite{UY}.
However, since $M$ is
an arbitrarily given Lorentzian manifold,
it is difficult to obtain an explicit expression for $Y$
in \eqref{eq:pde0}, unlike as in the proof 
in \cite{UY} for $n=2$.
So we newly prepare several propositions for proving 
the theorem. We may assume that $F$ is written 
in the form
\begin{align}
\label{eq:G1}
F(x_1,\dots,x_n)&=(f(x_1,\dots,x_n),x_1,\dots,x_n),\\
\label{eq:G2}
f(x_1,\dots,x_n)&=
a(x_n)+\sum_{i=1}^{n-1} b_i(x_n)x_i+
\sum_{1\le j\le k\le n-1} c_{j,k}(x_1,\dots,x_n)x_jx_k,
\end{align}
where $a,b_i,c_{j,k}$ are a $C^{r}$-function,
$C^{r-1}$-functions and $C^{r-2}$-functions, respectively. 
Moreover, since $f$ satisfies
\eqref{eq:c}, we have
\begin{equation}\label{eq:a-ini}
a(0)=0,\quad a'(0)=1,\quad b_i(0)=0\qquad (i=1,\dots,n-1),
\end{equation}
where ${}':=d/d x_n$.

We assume that $o$ is a degenerate light-like point.
By Lemma \ref{lem:nablaB}, we have
\begin{equation}\label{eq:deg}
 f_{x_n,x_1}(o)=\cdots=f_{x_n,x_{n-1}}(o)=0,
\end{equation}
which is equivalent to the conditions
\begin{equation}\label{eq:b-ini}
b'_i(0)=0\qquad (i=1,\dots,n-1).
\end{equation}

Let $\alpha$ be a non-negative integer.
We next assume that $F\in \XX^{r}_{\alpha}(M,\hat o)$ ($r\ge 3$).
Then there exists a $C^{r-2}$-function $\phi$ such that
\begin{equation}\label{eq:AB0}
\tilde A_F=0\qquad
(\tilde A_F:=A_F-(B_F)^{1+\alpha} \phi).
\end{equation}
Without loss of generality, we may assume that
$\phi$ is defined on $U_\sigma$.
Differentiating \eqref{eq:AB0} by the parameter 
$x_i$ ($i=1,\dots,n-1$), we have
\begin{equation}\label{eq:AB1}
(\tilde A_F)_{x_i}=
(A_F)_{x_i}-(B_F)^{\alpha+1}\phi_{x_i}
-(\alpha+1)(B_F)_{x_i}(B_F)^{\alpha}\phi=0.
\end{equation}

\begin{proposition}\label{prop:ODE-con}
Suppose that
$F$ satisfies {\rm (A1)} or {\rm (A2)} in 
Theorem \ref{thm:main}.
Substituting $x_1=\cdots =x_{n-1}=0$,
the two equations
\begin{equation}\label{eq:A000}
\left.\tilde A_F\right|_{(x_1,\dots ,x_{n-1})=
\vect{0}}=0,
\quad
\left.(\tilde A_F)_{x_i}\right|_{(x_1,\dots ,x_{n-1})=\vect{0}}=0
\qquad (i=1,\dots,n-1)
\end{equation}
induce a normal form of a system of
ordinary differential
equations 
with unknown $n$-functions
$a(x_n)$ and $b_i(x_n)$ $(i=1,\dots,n-1)$.
Moreover, the this system of ordinary equation 
satisfies the local Lipschitz condition. 
Furthermore, $a(x_n)$ and $b_i(x_n)$
as in the expression
\eqref{eq:G2} are just the solutions
with initial conditions \eqref{eq:a-ini} and \eqref{eq:b-ini}
of the system of this ordinary differential equations.
\end{proposition}

\begin{proof}
Since $f$ satisfies \eqref{eq:AB0}, we have
\begin{equation}\label{eq:pre0}
f_{x_n,x_n}=Y(t,\mb x ,f,\mb f_1,\dots,\mb f_{n}),
\end{equation}
where $Y$ is the function given
in the proof of Theorem \ref{thm:EA}.
Differentiating it by $x_i$, 
we get the following expression
\begin{equation}\label{eq:pre1}
f_{x_n,x_n,x_i}
=\tilde Y_i(t,\mb x ,f,\mb f_1,\dots,\mb f_{n},
(\mb f_1)_{x_i},\dots,(\mb f_{n})_{x_i}),
\end{equation}
where $\tilde Y_i$ is a $C^{r-3}$-function 
induced from $Y$ for each $i=1,\dots,{n-1}$.

For example, if $n=2$, \eqref{eq:pre0}
reduces to the relation
\begin{equation}\label{eq:pre02}
f_{y,y}=Y(t,x,y ,f,f_x,f_{xx},f_y,f_{yx}),
\end{equation}
where we set $x:=x_1,y=x_2$.
Then 
\begin{align*}
f_{y,y,x}&=
\frac{\partial Y}{\partial x}(t,x,y ,f,f_x,f_{xx},f_y,f_{yx})\\
&\phantom{aa}+f_{xx}
\frac{\partial Y}{\partial f_{1}}(t,x,y ,f,f_x,f_{xx},f_y,f_{yx})\\
&\phantom{aaaa}+f_{xxx}
\frac{\partial Y}{\partial f_{11}}(t,x,y ,f,f_x,f_{xx},f_y,f_{yx}) \\
&\phantom{aaaaaaa}+f_{yx}
\frac{\partial Y}{\partial f_{2}}(t,x,y ,f,f_x,f_{xx},f_y,f_{yx})\\
&\phantom{aaaaaaaaaa}+f_{yxx}
\frac{\partial Y}{\partial f_{21}}(t,x,y ,f,f_x,f_{xx},f_y,f_{yx}),
\end{align*}
where $f_1,f_{11},f_2,f_{21}$ 
are symbols as in \eqref{eq:multi}.

We return to the general case. We regard 
$c_{j,k}$ as fixed functions, 
and $a,\,\,b_i$ as unknown functions.
Substituting $x_1=\cdots=x_{n-1}=0$
into \eqref{eq:pre1}, 
we have a system of
ordinary differential equations
\begin{equation}
\label{eq:ode0}
\begin{cases}
a''&=Y(a,a',b_1,\dots, b_{n-1}),\\
b''_i&=Y_i(a,a',a'',b_1,\dots, b_{n-1},
b'_1,\dots, b'_{n-1}) \qquad (i=1,\dots,n-1)
\end{cases}
\end{equation}
of unknown functions $a(x_n)$ and $b_i(x_n)$,
where ${'}:=d/dx_n$, and each $Y_i$ ($i=1,\dots,n-1$)
is a $C^{r-3}$-function
of several variables induced by $\tilde Y_i$.
So if $r\ge 4$,
the system of ordinary differential equations
satisfies the local Lipschitz condition.
Here, \eqref{eq:ode0} is equivalent to 
\eqref{eq:A000}.

For example, if $n=2$, $\phi=0$ and $M=\R^3_1$,
\eqref{eq:ode0} reduces to \cite[(4.4) and (4.5)]{UY}.
\end{proof}

We next prove the following assertion:

\begin{proposition}\label{prop:B}
Suppose that $F_0=F_0(x_1,\dots,x_n)$ is written in the form
\begin{align}
\label{eq:F1}
F_0(x_1,\dots,x_n)&=(f_0(x_1,\dots,x_n),x_1,\dots,x_n),\\
\label{eq:F2}
f_0(x_1,\dots,x_n)&=x_n
+
\sum_{1\le j\le k\le n-1}
 c_{j,k}(x_1,\dots,x_n)x_jx_k,
\end{align}
where $\{c_{i,j}\}$ 
are $C^{r-2}$-functions defined at the origin $(r\ge 3)$.
Then 
\begin{align}
\label{eq:BG0a}
&B|_{(x_1,\dots,x_{n-1})=\vect{0}}=0, \\
\label{eq:BG0b}
& B_{x_i}|_{(x_1,\dots,x_{n-1})=\vect{0}}=0
\qquad (i=1,\dots,n-1)
\end{align}
hold, where $B:=B_{F_0}$.
Moreover, it holds
\begin{align}
\label{eq:AG0a}
& A|_{(x_1,\dots,x_{n-1})=\vect{0}}=0, \\
\label{eq:AG0b}
&A_{x_i}|_{(x_1,\dots,x_{n-1})=\vect{0}}=0
\qquad (i=1,\dots,n-1),
\end{align}
where $A:=A_{F_0}$.
\end{proposition}

\begin{proof}
In the expressions
\eqref{eq:G1} and
\eqref{eq:G2},
$F_0$ is the case that
\begin{equation}\label{eq:ab}
a(x_n)=x_n, \quad b_i(x_n)=0
\qquad (i=1,\dots,n-1).
\end{equation}
So we have
\begin{equation}\label{eq:f01}
f_0(0,\dots,0,x_n)=x_n,
\quad
(f_0)_{x_n}(0,\dots,0,x_n)=1,
\quad
(f_0)_{x_n,x_n}(0,\dots,0,x_n)=0
\end{equation}
and
\begin{equation}\label{eq:f02}
(f_0)_{x_i}(0,\dots,0,x_n)=(f_0)_{x_i,x_n}(0,\dots,0,x_n)=0
\qquad (i=1,\dots,n-1).
\end{equation}
Since all of the Christoffel symbols 
vanish along the curve $x_n\mapsto (x_n,0,\dots,0,x_n)$,
the formulas for $B,B_{x_i}$ and $A$ are 
completely the same as those in the 
case of $M=\R^{n+1}_1$,
so  \eqref{eq:f01} and \eqref{eq:f02} 
yield that
\begin{align*}
& B=1-(f_{x_1})^2-\cdots -(f_{x_n})^2=0, \\
&(B)_{x_i}=-2f_{x_n,x_i}=0 \qquad (i=1,\dots, n-1)
\end{align*}
along the $x_n$-axis, which proves
\eqref{eq:BG0a} and
\eqref{eq:BG0b},
where
$
(h)_{x_i}:={\partial h}/{\partial x_i}
$
($i=1,\dots, n$) for $h\in C^r_0(\R^n)$.
In particular, we have
$$
A=B \triangle f-\frac12\nabla B\star \nabla f=0
$$
along the $x_n$-axis, proving \eqref{eq:AG0a}.

From now on, we prove \eqref{eq:AG0b}:
Since $A=0$ along the $x_n$-axis,
we have $A_{x_n}=0$ along the $x_n$-axis. 
Thus, to prove the assertion, it is sufficient to
prove \eqref{eq:AG0b} for $i=1,\dots, n-1$.
By \eqref{eq:AF}, we have
$$
A_{x_i}=
\sum_{j,k=1}^n (\tilde s^{j,k})_{x_i}\tilde h_{j,k}
+\sum_{j,k=1}^n \tilde s^{j,k}(\tilde h_{j,k})_{x_i}
\qquad (i=1,\dots, n-1).
$$
Since all of the Christoffel symbols 
vanish along $(x_n,0,\dots,0,x_n)$,
$\tilde s^{i,j}$ 
and $\tilde h_{i,j}$ 
have the same expression as the case of $\R^{n+1}_1$
along the curve $\sigma$, and
we have (cf. \eqref{eq:b3} in Appendix B)
$$
(\tilde s^{j,k})_{x_i}=f_{x_i,x_k}f_{x_j}+f_{x_i,x_j}f_{x_k} 
=
\begin{cases}
f_{x_i,x_k} & \mbox{$j= n$}, \\
f_{x_i,x_j} & \mbox{$k= n$}, \\ 
0 & \mbox{otherwise}. 
\end{cases}
$$
On the other hand, 
\eqref{eq:f01} and \eqref{eq:f02} 
imply that $f_{x_i,x_n}=0$
for $i=1,\dots,n$ at $o$, with respect to the
Lorentzian metric $g$.
By \eqref{eq:b5} we have
$$
\tilde h_{n,k}=\tilde h_{j,n}=0 \qquad (j,k=1,\dots,n-1)
$$
along $\sigma$.
In particular, we have 
$$
\sum_{j,k=1}^n (\tilde s^{j,k})_{x_i}\tilde h_{j,k}=0
\qquad (i=1,\dots,n-1)
$$
along $\sigma$.
So it is sufficient to show that
\begin{equation}\label{eq:red2}
\sum_{j,k=1}^n \tilde s^{i,j}(\tilde h_{j,k})_{x_i}=0
\end{equation}
for $i=1,\dots,n-1$.
Since $B=0$ along the $x_n$-axis, 
$\tilde s^{i,j}$ vanishes
for $(i,j)\ne (n,n)$ (see (B.4) of Appendix B).
Moreover, the fact that $\tilde s^{n,n}=1$ along the $x_{n}$-axis
yields that the left-hand side of \eqref{eq:red2} is
equal to $(\tilde h_{n,n})_{x_i}$.
Then we have
$$
(\tilde h_{n,n})_{x_i}
=g(D_{\partial_{x_n}}(F_0)_{x_n},\tilde \nu)_{x_i} 
=
g(D_{\partial_{x_i}}D_{\partial_{x_n}}(F_0)_{x_n},\tilde \nu)
+
g(D_{\partial_{x_n}}(F_0)_{x_n},D_{\partial_{x_i}}\tilde \nu)
$$
for $i=1,\dots,n-1$.
Since
$$
D_{\partial_{x_n}}(F_0)_{x_n}=((f_0)_{x_n,x_n},0,\dots,0)
=\mb 0
$$
along the $x_n$-axis,
we have
\begin{align*}
(\tilde h_{n,n})_{x_i}
&=
g(D_{\partial_{x_i}}D_{\partial_{x_n}}(F_0)_{x_n},\tilde \nu)\\
&=
g(D_{\partial_{x_n}}D_{\partial_{x_i}}(F_0)_{x_n},\tilde \nu)+
g(R(\partial_{x_n},\partial_{x_i})(F_0)_{x_n},\tilde \nu),
\end{align*}
where $R$ is the curvature tensor of the connection $D$.
On the other hand, it can be easily checked that
$(F_0)_{x_n}$ is perpendicular to
$(F_0)_{x_i}$ ($i=1,\dots,n$).
Since vectors perpendicular to
$(F_0)_{x_i}$ ($i=1,\dots,n$) 
form a $1$-dimensional vector space at
each point of $F_0(x_n,0,\cdots,0,x_n)$, 
they are proportional to $\tilde \nu$ on the
$x_n$-axis.  So the second term of the
right-hand side vanishes because of symmetry of
the curvature tensor, and
$$
(\tilde h_{n,n})_{x_i}
=g(D_{\partial_{x_n}}D_{\partial_{x_i}}(F_0)_{x_n},\tilde \nu)
$$
holds for $i=1,\dots,n-1$.
Then $(\tilde h_{n,n})_{x_i}=0$
($i=1,\dots,n-1$),
because
$$
D_{\partial_{x_i}}(F_0)_{x_n}
=((f_0)_{x_i,x_n},0,\dots,0)=\mb 0
$$
along the $x_n$-axis, proving the assertion.
\end{proof}

\begin{proof}[Proof of Theorem A]
We set $\alpha=0$. 
When $M=\R^3_1$, the proof given in \cite{UY}
can apply for the case of $\alpha=0$.
Here we consider the case for general Lorentzian
manifold $M$ of dimension $n+1(\ge 3)$.
By Proposition \ref{prop:B},
$F_0$ satisfies $\tilde A_{F_0}=0$
along the $x_n$-axis,
where $\tilde A_{F_0}:=A_{F_0}-B_{F_0}\phi$.
Then we have $(\tilde A_{F_0})_{x_n}=0$
along the $x_n$-axis.
Thus
$$
\tilde A_{F_0}:=0,\qquad (\tilde A_{F_0})_{x_n}=0
$$
holds (cf. \eqref{eq:A000}).
By applying 
Proposition \ref{prop:ODE-con} for $F_0$,
$$
a_0(t)=t,\qquad b_{0,i}(t)=0\qquad (i=1,\dots,n-1)
$$
give the solution of the system of ordinary equations
\eqref{eq:ode0} with the initial condition
\begin{equation}\label{eq:ini1}
a(0)=0, \quad a'(0)=1,
\quad b_i(0)=0 \qquad (i=1,\dots,n-1), 
\end{equation}
and
\begin{equation}\label{eq:ini2}
b'_i(0)=0 \qquad (i=1,\dots,n-1). 
\end{equation}
The condition \eqref{eq:ini1}
follows if $o$ is a light-like point
(cf. \eqref{eq:c}), and the condition \eqref{eq:ini2}
follows if $o$ is a degenerate light-like point
(cf. Lemma \ref{lem:nablaB}).
Hence by applying Proposition \ref{prop:ODE-con} to $F$,
the same system of ordinary equations
\eqref{eq:ode0} is induced.
Since $r\ge 4$, this system of ordinary equations
satisfies the local Lipschitz condition.
Then the uniqueness of the solution implies that 
$a(t)=t$ and $b(t)=0$, that is, $F$ contains
a light-like geodesic $\sigma$
consisting of degenerate light-like points.
\end{proof}

\begin{remark}\label{new}
In the previous work \cite{UY}, the authors stated that
the assertion of Theorem A when $F$ is in
the class $\XX^3(\R^3_1)$. However,
the proof given in \cite{UY} is a special case
of the above proof, and so to prove 
\cite[Theorem 4.2]{UY}, 
we need to assume $F$ is $C^4$-differentiable.   
\end{remark}

As a consequence of Theorem \ref{thm:main}, 
we immediately get the following:

\begin{corollary}\label{cor:H}
Suppose that $F:U\to M$ 
is a $C^r$-immersion $(r\ge 3)$
whose mean curvature vector field
$\mb H_F$ extends 
as a $C^{1}$-vector field on $U$.
If $o\in U$ is a degenerate light-like point,
then $F(U)$ contains a light-like geodesic 
segment in $M$ passing through $\hat o(=F(o))$
consisting of only degenerate light-like points.
\end{corollary}
\begin{proof}
If $\mb H_F$ is $C^1$-differentiable, then 
$F$ belongs to the class $\XX^{3}_{1}(M,\hat o)$. Since 
$\XX^{3}_1(M,\hat o)\subset \XX^{3}_0(M,\hat o)$,
the assertion follows immediately  from Theorem A.
\end{proof}

\begin{example}\label{ex:X1}
We set
 \[
    F_1(x,y):=\bigl(f_1(x,y),x,y\bigr),\qquad
    f_1(x,y):= y+ x^2 + x^3 + yx^4.
 \]
 Then
 \begin{align*}
    A_{F_1}&=2x^4(6+6x+4yx^2 +7x^4 + 9x^5+10y x^6 y),\\
    B_{F_1}&=-x^2(4+12x + (11+16y)x^2+24y x^3 + 16y^2 x^4 + x^6).
 \end{align*}
 So $F_1\in \XX^\omega_{1,\phi_1}(\R^3_1)$
 and $(0,0)$ is a degenerate light-like point.
 The mean curvature function $H_{F_1}$ is bounded on 
 a neighborhood of $(0,0)$ but not real analytic.
But the another mean curvature $\hat H_F$ (cf.
\eqref{eq:HF2}) is real analytic.
\end{example}

\begin{example}\label{ex:X2}
We next set
 \[
    F_2(x,y):=\bigl(f_2(x,y),x,y\bigr),\qquad
    f_2(x,y):= y-(1+y)x^3 -y^3 x^4.
 \]
 Then
 \begin{align*}
    A_{F_2}=6x^4(1+h_1(x,y) x),\qquad
    B_{F_2}=x^3(2 + h_2(x,y) x),
 \end{align*}
 where $h_1(x,y)$ and $h_2(x,y)$ are polynomials in $x,y$.
 So 
$F_2\in \XX^\omega_{0,\phi_2}(\R^3_1)$
 and $(0,0)$ is a degenerate light-like point.
 The mean curvature function $H_{F_2}$ is unbounded at $(0,0)$.
\end{example}

As shown in Corollary \ref{cor:ZYX},
there are many surfaces with 
degenerate light-like points
such that their mean curvature vector field
$\mb H_F$ is real analytic on $U$.

\begin{corollary}\label{cor:isolated}
Suppose that $F:U\to M$ 
is an immersion satisfying
$F\in \XX^{3}_0(M,\hat o)$.
Then any light-like points of $F$
are not isolated.
\end{corollary}

If the light-like point $o$ is non-degenerate, 
then $B_F$ changes sign, that  is, $F$ changes 
its causal type (i.e. $U_+,U_-$ are both
non-empty). So the following corollary is 
also obtained.

\begin{corollary}
Suppose that
$F\in \XX^{r}_0(M,\hat o)$ $(r\ge 4)$.
If $o$ is a light-like point at which
$U_+$ and $U_-$ do not simultaneously accumulate to 
$o$, 
then $F(U)$ contains a segment of
a light-like geodesic passing through $\hat o(=F(o))$.
\end{corollary}

Finally, we give here a proof of Corollary B 
stated in the introduction:

\begin{proof}[Proof of Corollary B]
Let $F:U\to M$ be a $C^4$-differentiable light-like immersion.
By Theorem~A, at each point $p\in U$,
there exists a light-like geodesic segment
$L$ passing through $F(p)$
such that $L\subset F(\Sigma_F)$ (cf. (0.5)). 
Since the Lorentzian metric of $M$
is of index $1$,
the induced metric on
$U$ by $F$ is of rank $n-1$.
So, the tangential direction of $L$ just corresponds
to the null direction at $p$.
Thus, we get the assertion.
\end{proof}

\section{Properties of curves consisting of 
degenerate light-like points}

We firstly prepare the following: 

\begin{proposition}\label{prop:weak-version}
Let $F:U \mapsto M$ be a $C^r$-differentiable $(r\ge 2)$
immersion of a domain $U\subset \R^n$
into  a Lorentzian $(n+1)$-manifold $M$.
Suppose that $\gamma:I\to U$ is a null curve
$($cf. \eqref{eq:null}$)$ satisfying
$B_F=0$ and
$\nabla B_F=\mb 0$
along the curve $\gamma$.
Then 
\begin{equation}\label{eq:acc}
\hat\gamma''(t):=D_{\gamma'(t)}\hat\gamma'(t)
\end{equation}
 is proportional to
 $\hat\gamma'(t)$,
where $\hat \gamma:=F\circ \gamma$, and
 $D$ is the Levi-Civita connection associated with
 $g$. In particular, if 
 $\gamma(I) (\subset U)$
 consists of degenerate light-like points,
 then $\hat\gamma(t)$ is a light-like geodesic in 
$M$,
by a suitable change of the parameter $t$.
\end{proposition}

\begin{proof}
We can take vector fields
$
X_1(t),\dots,X_{n-1}(t)\in T_{\gamma(t)}U
$
along the curve $\gamma(t)$
such that
$$
\gamma'(t)\,\,,X_1(t)\,\,,\dots,\,\,X_{n-1}(t)
$$
forms a basis of $T_{\gamma(t)}U$ at each point $t\in I$.
We let $W_t$ be the subspace of
$T_{\gamma(t)}U$  spanned by $X_1(t),\dots,X_{n-1}(t)$.
Since $g$ is of signature $(n-1,1)$ and
$\hat \gamma'(t)$ points in the null-direction, 
the restriction of the first fundamental form
$ds^2$ to $W_t$ is positive definite.
By the Schmidt-orthogonalization, we can make
vector fields 
$$
E_1(t),\,\,\dots,\,\,E_{n-1}(t)\in T_{\gamma(t)}U
$$
along the curve $\gamma(t)$
that give an orthonormal basis of $W_t$ 
at each $t\in I$.
We now take a Riemannian metric $\langle \,,\, \rangle$
on $U$ such that
$$
\gamma'(t),\,\,E_1(t)\,\,,\dots,\,\,E_{n-1}(t)
$$
gives an orthonormal basis of $T_{\gamma(t)}U$.
We then take a geodesic tubular coordinate 
neighborhood $(y_1,\dots,y_{n})$  along $\gamma(t)$
with respect to the
Riemannian metric $\langle \,,\, \rangle$
such that
\[
 \left(\partial_{y_n}\right)_{\gamma(t)}=\gamma'(t),\quad
 \left(\partial_{y_i}\right)_{\gamma(t)}=E_{i}(t)
\qquad (i=1,\dots,n-1),
\]
where $\partial_{y_j}:=\partial/\partial y_j
\,\,(j=1,\dots,n)$.
Then $\gamma(t)=(0,\dots,0,t)$ parametrizes the $y_n$-axis.
We set
$$
s_{i,j}^{}:=g(dF(\partial_{y_i}),dF(\partial_{y_j}))
(=ds^2(\partial_{y_i},\partial_{y_j}))
\qquad (i,j=1,\dots,n).
$$
Then, by our construction of the 
coordinates  $(y_1,\dots,y_n)$, we have
\begin{equation}\label{eq:dij-add}
s_{i,j}^{}=
\begin{cases}
\delta_{i,j} & \mbox{$i,j=1,\dots,n-1$},\\
0            &  \mbox{otherwise}
\end{cases}
\end{equation}
along $\gamma(t)$.
Since $s_{n,n}^{}=s_{n,j}^{}=0$ along the $y_n$-axis (i.e. along $\gamma(t)$), 
we have
\begin{equation}\label{eq:D0}
 0=(s_{n,n}^{})_{y_n}=
 g(dF(\partial_{y_n}),dF(\partial_{y_n}))_{y_n}
 =2g(D_{\partial_{y_n}}dF(\partial_{y_n}),dF(\partial_{y_n})),
\end{equation}
and
\begin{align}\label{eq:000}
0&=(s_{n,i}^{})_{y_n}=g(dF(\partial_{y_n}),dF(\partial_{y_i}))_{y_n} \\
\nonumber
&=g(D_{\partial_{y_n}}dF(\partial_{y_n}),dF(\partial_{y_i}))
+g(dF(\partial_{y_n}),D_{\partial_{y_n}}dF(\partial_{y_i}))\\
\nonumber
&=g(D_{\partial_{y_n}}dF(\partial_{y_n}),dF(\partial_{y_i}))
+g(dF(\partial_{y_n}),D_{\partial_{y_i}}dF(\partial_{y_n}))\\
\nonumber
&=g(D_{\partial_{y_n}}dF(\partial_{y_n}),dF(\partial_{y_i}))
+\frac{1}{2} g(dF(\partial_{y_n}),dF(\partial_{y_n}))_{y_i}\\
\nonumber
&= g(D_{\partial_{y_n}}dF(\partial_{y_n}),dF(\partial_{y_i}))
+\frac{1}{2}(s_{n,n}^{})_{y_i}
\end{align}
holds along the $y_n$-axis for $i=1,\dots,n-1$.
We now suppose that $\nabla B_F=\mb 0$ 
along the $y_n$-axis.
Then, we have
$$ 
0=B_{y_i}
=\frac{\partial}{\partial y_i}
\det\pmt{
s_{1,1}^{} &\cdots & s_{1,n}^{} \\
\vdots &\ddots & \vdots \\
s_{n,1}^{} &\cdots & s_{n,n}^{} 
} 
=\det\pmt{
s_{1,1}^{} &\cdots & s_{1,n}^{} \\
\vdots &\ddots & \vdots \\
(s_{n,1}^{})_{y_i}^{} &\cdots & (s_{n,n}^{})_{y_i} 
}
=(s_{n,n}^{})_{y_i}^{},
$$
where $B:=B_F$. By this and
\eqref{eq:000}, we have
\begin{equation}\label{eq:Di}
g(D_{\partial_{y_n}}dF(\partial_{y_n}),
dF(\partial_{y_i}))=0
\qquad (i=1,\dots,n-1)
\end{equation}
along the $y_n$-axis.
We newly take a vector 
field $E_0(t)\in T_{\hat \gamma(t)}M$
along the $y_n$-axis such that
$E_0$ is perpendicular to $dF(\partial_{y_j})$ ($j=1,\dots,n-1$)
and 
$$
E_0,dF(\partial_{y_n}),\,\,
dF(\partial_{y_1}),\,\dots,\,\,dF(\partial_{y_{n-1}}),\,\,
$$
are linearly independent at each $t\in I$.
Here
$D_{\partial_{y_n}}dF(\partial_{y_n})$
can be expressed as a linear combination
$$
D_{\partial_{y_n}}dF(\partial_{y_n})
=a(t) E_0+b(t)\,dF(\partial_{y_n})+\tilde{\vect{w}},\qquad
\tilde{\vect{w}}:=\sum_{j=1}^{n-1} c_j(t)
dF(\partial_{y_j}), 
$$
at each $t\in I$, where
$a(t),b(t),c_1(t),\dots,c_{n-1}(t)$ are $C^{r-2}$-functions on $I$.
By \eqref{eq:Di}, 
we have
$$
0=g(D_{\partial_{y_n}}dF(\partial_{y_n}),\tilde{\vect{w}})=
g(\tilde{\vect{w}},\tilde{\vect{w}})
=ds^2(\vect{w},\vect{w}),
$$
where
$
\vect{w}:=\sum_{j=1}^{n-1} c_j
\partial_{y_j}. 
$
Since $ds^2$ is positive definite
on the subspace spanned by $\{\partial_{y_i}\}_{i=1}^{n-1}$,
we can conclude that $\tilde{\vect{w}}=\vect{0}$.
So 
$$
D_{\partial_{y_n}}dF(\partial_{y_n})
=a E_0+b\,dF(\partial_{y_n})
$$
holds along the $y_n$-axis. 
By \eqref{eq:D0}, we have
$$
0=g(D_{\partial_{y_n}}
dF(\partial_{y_n}),dF(\partial_{y_n}))
=a\,g(E_0,dF(\partial_{y_n})).
$$
Since 
$g(dF(\partial_{y_n}),dF(\partial_{y_i}))=0$ for $i=0,\dots,n-1$
along the $y_n$-axis, the non-degeneracy of the metric $g$
yields that $g(E_0,dF(\partial_{y_n}))
\ne 0$ for each $t$. In particular,
$a(t)=0$ for each $t$.
Thus
$
\hat\gamma''(t)\left(=D_{\partial_{y_n}}
dF(\partial_{y_n})\right)
$
is proportional to $\hat\gamma'(t)$.
\end{proof}

As an application of 
Proposition \ref{prop:weak-version},
we reprove Corollary B in the introduction
under the $C^2$-differentiability.
:

\begin{proof}[An alternative proof of Corollary B]
Let $F:U\to M$ be a $C^2$-differentiable light-like 
immersion. As pointed out in the previous proof
of Corollary B at the end of Section 4,
the induced metric on
$U$ by $F$ is of rank $n-1$.
So, we can take a smooth vector field that points
in the 
null direction at each point on $U$.
Then the integral curves of this foliation
consist of null curves.
By Proposition \ref{prop:weak-version},
the image of these curves are light-like geodesics.
\end{proof}

From now on, we consider the case $n=2$:

\begin{proof}[Proof of Corollary~\ref{cor:main}]
 The assertion (b) is just the statement of 
 Theorem \ref{thm:main}  for $n=2$.
 So it is sufficient to show the assertion (a), which immediately 
 follows from the following proposition:
\end{proof}

\begin{proposition}\label{prop:n=2}
Let $M$ be a Lorentzian $3$-manifold, and
$F:U\to M$ a $C^r$-differentiable immersion
of a domain $U\subset\R^2$
 belonging to the class $\tilde {\mc I}^r_L
(M,\hat o)$ for $r\geq 3$.
 Let $\gamma:I\to U$ be a regular
 curve such that $B_F$ vanishes 
and $\nabla B_F\ne (0,0)$ along $\gamma$.
 Then $\gamma$ is null, that is,
 $ds^2(\gamma'(t),\gamma'(t))=0$ holds.
 Moreover, $\hat \gamma''(t)$ is linearly independent of
$\hat \gamma'(t)$ for each $t\in I$, where
$\hat \gamma=f\circ \gamma$
and
$\hat \gamma''(t):=D_{\gamma'(t)}\hat\gamma'(t)$. 
\end{proposition}

\begin{proof}
We set $B:=B_F$ and
$x:=x_1$ and $y:=x_2$ for the sake of simplicity.
We set $\gamma(t)=(x(t),y(t))$. 
Differentiating 
the relation
$B(x(t),y(t))=0$, we have
\begin{equation}\label{eq:Bxy0}
x'(t)B_x(\gamma(t))+y'(0)B_y(\gamma(t))=0.
\end{equation}
Without loss of generality, we may assume that $o=\gamma(t)$.
Moreover, we may assume that $F\in {\mc I}_L^r(M,\hat o)$.
Then we can write $F(x,y)=(f(x,y),x,y)$ and
\begin{equation}
 f_x(\gamma(t))=0,\qquad f_y(\gamma(t))=1.
\end{equation}
By \eqref{eq:Bxy0},
$(x'(t),y'(t))$ is proportional
to $(-B_y(\gamma(t)),B_x(\gamma(t)))(\ne (0,0))$.
By Lemma \ref{lem:nablaB}, we have
$B_{y}(\gamma(t))=0$, and the vector $(x'(t),y'(t))$ 
is proportional $(0,1)=\partial_y$.
Since 
\[
  ds^2_{\gamma(t)}(\partial_y,\partial_y)=s_{2,2}^{}(\gamma(t))=
   1-f_{y}(\gamma(t))^2=0,
\]
the vector $(x'(t),y'(t))$ points in the null direction,
that is, 
$$F_x(\gamma(t)) x'(t)+F_y(\gamma(t))y'(t)
$$ 
gives a light-like vector.
Thus $\gamma(t)$ is a null curve.

We next prove the second assertion.
Since $\nabla B_F\ne (0,0)$ along $\gamma$,
the vector field $\hat \gamma''(t)$ 
is not proportional to $\hat\gamma'(t)$, by
Proposition \ref{prop:weak-version}.
\end{proof}

\section{Proof of Theorem D}

Let $F:U\to M$
be a $C^\omega$-immersion 
with a light-like point $o\in U$
whose mean 
curvature function $H_F$ satisfies the assumption of Theorem D.
Then $F$ belongs to the class $\XX^\omega_{1/2}(M,\hat o)$.
In this case, $\sqrt{|B_F|}$ may not be smooth in general,
and the ordinary differential equation
\eqref{eq:ode0} may not satisfy the 
local Lipschitz condition.
However, we can modify the proof of Theorem A as follows: 
Under the assumption of Theorem D,
there exists $\phi\in C^{\omega}_0(\R^n)$
such that $\phi=n H_F$ holds on $U\setminus \Sigma_F$
(cf. (0.5)) and
\begin{equation}\label{eq:AB1/2}
A_F-\varphi |B_F|^{3/2}=0, \quad |\phi|>0
\end{equation}
hold on $U$. In this case, $F$ never changes type, by 
\cite[Theorem 1.1]{HKKUY}.
\begin{lemma}
We set
$$
\delta(x_n):=\sqrt{|B_{F}|}\biggr|_{(x_1,\dots,x_{n-1})=\vect{0}}.
$$
Then $\delta$ is a real analytic function of $x_n$.
\end{lemma}

\begin{proof}
Since we know that $B_F$ never changes sign,
we can write
\begin{equation}\label{eq:cmcBF}
B_{F}|_{(x_1,\dots,x_{n-1})=\vect{0}}=c (x_n)^{2m}(1+k(x_n)),
\end{equation}
where $c$ is a non-zero constant, $m$ is a positive integer and
$k(x_n)$ is a real analytic function of
one variable.
By \eqref{eq:AB1/2},
we know that $\delta(x_n)^{3}=A_F/\phi$ is 
also real analytic.  
So $m$ is an even number, and we can write
$$
\delta(x_n)=|c|(x_n)^{m}\sqrt{1+k(x_n)},
$$
proving the assertion.
\end{proof}

For the sake of the readers' convenience, 
we first prove Theorem D for $M=\R^{3}_1$
under the assumption that the mean curvature
function $H_F$ is constant:

\begin{proof}[{$($Proof of Theorem D for $M=\R^{3}_1)$}]
Let $F:U\to \R^3_1$ be a real analytic immersion
having mean curvature  $H_F=C/n(\ne 0)$ 
in the form $F(x,y)=(f(x,y),x,y)$ such that $(0,0)$ is a light-like 
point.
Without loss of generality, we may assume (cf.\ \eqref{eq:G2}) that
\[
    f(x,y) = a(y) + b(y)x + h(x,y) x^3,
\]
where $a(y)$, $b(y)$ and $h(x,y)$ are real analytic functions with
\[
   a(0)=0,\qquad a'(0)= 1, \qquad b(0) = 0.
\]
Because of \eqref{eq:cmcBF}, the origin is a degenerate
light-like point.
Hence $b'(0)=0$ holds (cf.\ Lemma~\ref{lem:nablaB}).
We set
$$
\tilde A:=A_F -C B_F |B_F|^{1/2} (= A_F-CB_F\delta).
$$
Then we have
\begin{align}
\label{eq:Hc}
  0(=\tilde A|_{x=0})=&
  C (1-(a')^2-b^2)\delta+h_0 
  \left(1-(a')^2\right)+\left(1-b^2\right) a''+2 b 
  a' b',  \\
\label{eq:Hc2}
  0(=\tilde A_x|_{x=0})=&
  -2 b h_0 a''+3C 
  \delta  \left(b h_0+a' b'\right)\\
\nonumber
  &\phantom{aa}+2 b h_2 a'
-3h_1 \left((a')^2-1\right) -
  \left(b^2-1\right) b''  +2 b (b')^2,
\end{align}
where
 \begin{alignat*}{3}
  h_0(y)&:=h(0,y),
  \quad &
  h_1(y)&:=h_x(0,y),
  \quad &
  h_2(y)&:=h_y(0,y).
\end{alignat*}
Then \eqref{eq:Hc} and \eqref{eq:Hc2}
 can be considered as a system of
ordinary differential equations with
unknown functions $a,b$, 
regarding $\delta,\,\,h_j$ ($j=0,1,2$)
as given functions. Here, the crucial point is 
that we think of $\delta$ as a given function.
Namely, we consider the following 
system of
ordinary differential equations with
unknown functions $u,v$, as follows:
\begin{align}
\label{eq:Hcc}
  0(=\tilde A|_{x=0})=&
  C (1-(u')^2-v^2)\delta+h_0 
  \left(1-(u')^2\right)+\left(1-v^2\right) u''+2 v 
  u' v',  \\
\label{eq:Hcc2}
  0(=\tilde A_x|_{x=0})=&
  -2 v h_0 u''+3C 
  \delta  \left(v h_0+u' v'\right)\\
\nonumber
  &\phantom{aa}+2 v h_2 u'
-3h_1 \left((u')^2-1\right) -
  \left(v^2-1\right) v''  +2 v (v')^2, 
\end{align}
where $\delta:=\sqrt{1-(a')^2-b^2}$. 
The fact that $\delta$ is real analytic
(cf. Lemma 5.1)
yields that we can apply the uniqueness of
the system 
\eqref{eq:Hcc},
\eqref{eq:Hcc2}
of ordinary differential equations 
with initial data $u(0)=0$, $u'(0)=1$ and $v(0)=0$,
$v'(0)=0$.
So we can conclude that
$(u(y),v(y))=(y,0)$.
Since $(u,v)=(a,b)$ is also a solution of
the same system of equations,
we have $(a(y),b(y))=(y,0)$.
\end{proof}

\begin{proof}[{$($Proof of Theorem D for the general case$)$}]
We set $F_0$ as in
\eqref{eq:F1} and
\eqref{eq:F2}.
By Proposition \ref{prop:B}, $F_0$ satisfies 
\begin{equation}\label{eq:5-1}
\tilde A_{F_0}=0, \qquad (\tilde A_{F_0})_{x_i}=0
\qquad (i=1,\dots,n-1)
\end{equation}
along the $x_n$-axis,
where $\tilde A_{F_0}:=A_{F_0}-\phi B_{F_0}\sqrt{|B_{F_0}|}$.

On the other hand,
\begin{equation}\label{eq:5-2}
\tilde A_{F}=0, \qquad (\tilde A_{F})_{x_i}=0
\qquad (i=1,\dots,n-1)
\end{equation}
correspond to
\eqref{eq:AB0} and
\eqref{eq:AB1}, respectively.
Namely,
\begin{align}\label{eq:AB1a}
&A_F-B_F|B_F|^{1/2}\phi=0\\
\label{eq:AB1b}
&(\tilde A_F)_{x_i}=
(A_F)_{x_i}-B_F|B_F|^{1/2}\phi_{x_i}
-\frac32\tilde\delta \phi(B_F)_{x_i}=0
\quad (i=1,\dots,n-1),
\end{align}
where
$$
\tilde\delta(x_1,\dots,x_n):=\sqrt{|B_F(x_1,\dots,x_n)|}.
$$
Then \eqref{eq:AB0} and \eqref{eq:AB1b}
induce the following two equalities, respectively:
\begin{align}\label{eq:fin1}
& A_{F}|_{(x_1,\dots,x_{n-1})=\vect{0}}-\hat \phi \delta
  B_{F}|_{(x_1,\dots,x_{n-1})=\vect{0}}
=0, \\
\label{eq:fin2}
&(A_F)_{x_i}|_{(x_1,\dots,x_{n-1})=\vect{0}}
-\hat \phi_{i}
\delta B_F |_{(x_1,\dots,x_{n-1})=\vect{0}} \\
\nonumber
& \phantom{aaaaaaaaaaaaaaaaaaaaaa}
-\frac32\delta \hat \phi(B_F)_{x_i}|_{(x_1,\dots,x_{n-1})=\vect{0}}=0,
\end{align}
where 
$$
\hat \phi(x_n):=\phi(0,\dots,0,x_n),\quad
\hat \phi_i(x_n):=\phi_{x_i}(0,\dots,0,x_n)
\qquad (i=1,\dots,n-1)
$$
and $\delta(x_n)=\hat\delta(0,\dots,0,x_n)$.
Since the expressions of two functions
of one variable
$$
A_{F}|_{(x_1,\dots,x_{n-1})=\vect{0}},\qquad 
B_{F}|_{(x_1,\dots,x_{n-1})=\vect{0}}
$$
contain $a(x_n),a'(x_n)$ and $b_i(x_n),b'_i(x_n)$
($i=1,\dots,n-1$),
 \eqref{eq:fin1} and \eqref{eq:fin2}
can be considered as  a
system of second order ordinary
differential equations 
with unknown functions
$a(x_n)$ and $b_i(x_n)$
($i=1,\dots,n-1$).
Here, the function $\delta(x_n)$ is considered
as a given function.
By
\eqref{eq:5-1},
$$
a_0(t):=t,\qquad b_{0,i}(t):=0\qquad (i=1,\dots,n-1)
$$
give a solution of this new system of ordinary equations
with the initial condition
\begin{equation}\label{eq:ini1sec}
a(0)=0, \quad a'(0)=1,
\quad b_i(0)=0 \qquad (i=1,\dots,n-1), 
\end{equation}
and
\begin{equation}\label{eq:ini2sec}
b'_i(0)=0 \qquad (i=1,\dots,n-1). 
\end{equation}
The condition \eqref{eq:ini1sec}
implies that $o$ is a light-like point
(cf. \eqref{eq:c}), and the condition \eqref{eq:ini2sec}
implies that $o$ is a degenerate light-like point
(cf. Lemma \ref{lem:nablaB}).
Since $\delta$ is a real analytic function on a neighborhood of
$o$, this system of ordinary differential equation
satisfies the Lipschitz condition.
Hence the uniqueness of the solution implies that 
$a(t)=t$ and $b_i(t)=0$ ($i=1,\dots,n-1$), 
that is, $F$ contains
a light-like geodesic consisting of degenerate light-like points.
\end{proof}

\begin{remark}\label{rem:HF}
In the above proof of Theorem D, we may use 
$\hat H_F:=\op{sign}(B_F)H_F$ instead of $H_F$
(cf.~Remark \ref{rmk:1-2}).
In fact, $B_F$ never changes sign, because of 
\cite[Theorem 1.1]{HKKUY},
and $\hat H_F=H_F$ or $\hat H_F=-H_F$
holds on $U$.
\end{remark}

In the authors' previous works,
several examples of zero mean curvature surfaces
having light-like lines
in $\R^3_1$
were given.
Also several examples having bounded
non-zero constant mean curvature 
with light-like geodesics were given in
\cite{HKKUY}.
As other examples, several space-like 
surfaces with non-zero constant mean curvature $1$
containing light-like lines
in the de Sitter 3-space $S^3_1$
have recently been found in \cite{FKKRUYY}.
Finally, we give here an example with non-constant mean
curvature function satisfying the assumption of
Theorem D:

\begin{example}\label{ex:X3}
We set (see also Examples \ref{ex:X1}, \ref{ex:X2})
\[
  F_3(x,y):=\bigl(f_3(x,y),x,y\bigr),\quad
  f_3(x,y):=y+x^3 + x^4 + yx^5.
\]
Then 
\begin{align*}
  A_{F_3} &= 2x^6(9+8x + 5yx^2 + 12 x^5 + 14x^6+15yx^7),\\
  B_{F_3} &= -x^4 (9+26x+2(8+15y)x^2 + 40yx^3 + 25y^2 x^4+x^6).
\end{align*}
So $F_3\in \XX^\omega_0(\R^3_1)$ and  
satisfies also the assumption of Theorem D.
The mean curvature function $H_{F_3}$ 
is real analytic.
\end{example}

\appendix
\section{Fermi coordinate systems along light-like geodesics}
\label{app:imp}

Let $(M,g)$ be a Lorentzian $(n+1)$-manifold.
In this appendix, we prove the following assertion:

\begin{proposition}\label{prop:A}
Let $I=(a,b)$ be a closed interval,
and let $\sigma:I\to M$ be 
a light-like geodesic.
Then there exists a local diffeomorphism
$(\epsilon>0$ is a constant$)$
\begin{align*}
&\Phi:\biggl\{(x_0,\dots,x_n)\in \R^{n+1}\,;\,
x_0,x_n\in \bigl[\frac{a+\epsilon}{\sqrt{2}},\frac{b-\epsilon}{\sqrt{2}}\bigr],
\\
&\phantom{aaaaaaaaaaaaaaaaaaaaaaaaaa}
\,\, |x_i|<\epsilon\,\, (i=1,\dots,n-1)
\biggr\} \to M
\end{align*}
 
such that 
\begin{enumerate}
\item[{\rm (a1)}] $\Phi(t,0,\dots,t)=\sigma(t)$
holds for $t\in  \left[{a}/{\sqrt{2}},{b}/{\sqrt{2}}\right]
$, 
\item[{\rm (a2)}] regarding $(x_0,\dots,x_n)$
as a local coordinate system by $\Phi$, 
$g_{0,0}^{}=-1$,\,\, $g_{0,i}^{}=0$ along $\sigma$
and $g_{j,k}^{}=\delta_{j,k}$ hold
for $i,j,k=1,\dots,n$ along $\sigma$,
where $g_{j,k}^{}:=g(\partial_{x_j},\partial_{x_k})$, and
\item[{\rm (a3)}] all of the Christoffel symbols 
vanish along $\sigma$. As a consequence,
all derivatives $\partial g_{\alpha,\beta}^{}/\partial x_\gamma$
$(\alpha,\beta,\gamma=0,\dots,n)$ vanish  along $\sigma$.
\end{enumerate}
\end{proposition}
\begin{proof}
We can take vectors 
$\vect{e}_0,\dots,\vect{e}_n\in T_{\sigma(a)}M$ such that
\begin{enumerate}
\item $\vect{e}_0,\vect{e}_1,\dots,\vect{e}_n$ ($\vect{e}_0:=\sigma'(a)$)
gives a basis of $T_{\sigma(a)} M$,
\item $g(\vect{e}_0,\vect{e}_0)=
g(\vect{e}_1,\vect{e}_1)=0$ 
and $g(\vect{e}_0,\vect{e}_1)=-1$,
\item $g(\vect{e}_0,\vect{e}_i)=g(\vect{e}_1,\vect{e}_i)=0$ 
for $i=2,\dots,n$,
\item $g(\vect{e}_j,\vect{e}_k)=\delta_{j,k}$ for
$j,k=2,\dots,n$. 
\end{enumerate}
We let $E_\alpha(t)$ ($a\le t\le b$,\,\, $\alpha=0,\dots,n$) 
be the parallel vector field along $\sigma(t)$ such that
$$
E_\alpha(a)=\vect{e}_\alpha \qquad (\alpha=0,\dots,n).
$$
Since $\sigma$ is a geodesic, $\sigma'(t)=E_0(t)$
holds for $t\in I$.
We then set
$$
\Phi:\R\times \R^n\ni (t,y_1,\dots,y_n)\mapsto 
\op{Exp}_{\sigma(t)}\left(y_1E_1(t)+\cdots +y_n E_n(t)\right)
\in M,
$$
where $\op{Exp}_p$ is the exponential map
at $p\in M$ with respect to the 
Lorentzian metric $g$.
Then there exists $\epsilon>0$ such that
the restriction of $\Phi$ to the open domain
$$
\left\{(y_0,y_1\dots,y_n)\in \R\times \R^n;\, a<y_0<b,\quad |y_i|<\epsilon 
\,\,\,  (i=1,\dots,n)\right\}
$$
gives a diffeomorphism, where $y_0:=t$.
By definition, we have
$$
E_{\alpha}(t)
=\left(\frac{\partial}{\partial y_{\alpha}}\right)_{\sigma(t)}\qquad 
(\alpha=0,\dots,n),
$$
and
\begin{itemize}
\item $g(\partial_{y_0},\partial_{y_0})=
g(\partial_{y_1},\partial_{y_1})=0$ 
and $g(\partial_{y_0},\partial_{y_1})=-1$,
\item $g(\partial_{y_0},\partial_{y_i})=g(\partial_{y_1},\partial_{y_i})=0$ 
for $i=1,\dots,n$ holds along $\sigma$,
\item $g(\partial_{y_j},\partial_{y_k})=\delta_{j,k}$ for
$j,k=2,\dots,n$ along $\sigma$. 
\end{itemize}
Let $D$ be the Levi-Civita connection of $M$.
Then the Christoffel symbols with respect to 
this local coordinate system $(y_0,\dots,y_n)$
are defined by
$$
D_{\partial_{y_\alpha}}\partial_{y_\beta}=
\sum_{\gamma=0}^n\Gamma_{\alpha,\beta}^\gamma \partial_{y_\gamma}
\qquad (\alpha,\beta=0,\dots,n).
$$
We would like to show that
all of the  Christoffel symbols $\Gamma_{\alpha,\beta}^\gamma$
($\alpha,\beta,\gamma=0,\dots,n$)
with respect to the local coordinates $(y_0,\dots,y_n)$
vanish along the $y_0$-axis (i.e. along the curve $\sigma$). 
We fix $(a_1,\dots,a_n)\in \R^n\setminus\{\mb 0\}$, and
consider a curve 
$
c(t):=(0,a_1t,\dots,a_n t)
$
in $M$.
Then, by our definition of the local coordinate system $(y_0,\dots,y_n)$,
this curve $c(t)$ gives a geodesic on $M$.
So $c(t)=(c_0(t),\dots,c_n(t))$ satisfies
$$
c''_\gamma(t)+\sum_{\alpha,\beta=0}^n
\Gamma_{\alpha,\beta}^\gamma(c(t)) c'_\alpha(t) c'_\beta(t)=0
\qquad (\gamma=0,\dots,n).
$$
Since $c'_0(t)=c''_0(t)=\cdots =c''_n(t)=0$, and
$c'_i(t)=a_i$ for $i=1,\dots,n$, this
reduces to
$$
\sum_{i,j=1}^n
\Gamma_{i,j}^\gamma(c(t))a_i a_j=0
\qquad (\gamma=0,\dots,n).
$$
If we set $a_i=a_j=1$ and other $a_k=0$ for $k\ne i,j$,
then we have
$$
\Gamma_{i,j}^\gamma(c(t))+\Gamma_{j,i}^\gamma(c(t))=0
\qquad (i,j=1,\dots,n,\,\,\gamma=0,\dots,n).
$$
Since the connection is torsion free,
$\Gamma_{\alpha,\beta}^\gamma
=\Gamma_{\beta,\alpha}^\gamma$
holds ($0\le \alpha,\beta,\gamma\le n$), 
and we get
\begin{equation}\label{eq:C1}
\Gamma_{i,j}^\gamma(c(t))=0
\qquad (i,j=1,\dots,n,\,\,\gamma=0,\dots,n).
\end{equation}
On the other hand, since $\sigma$ is a geodesic
and $E_0(t),\dots, E_n(t)$ are parallel vector fields
along $\sigma$, we have
\begin{equation}\label{eq:C2}
\Gamma_{0,\beta}^\gamma(c(t))=0
\qquad (\beta,\gamma=0,\dots,n).
\end{equation}
By \eqref{eq:C1} and \eqref{eq:C2}
together with the property 
$\Gamma_{\alpha,\beta}^\gamma
=\Gamma_{\beta,\alpha}^\gamma$,
we can conclude that
all of the Christoffel symbols along the 
curve $\sigma$ vanish.

We now set
\begin{equation}\label{eq:y-to-x}
x_0:=\frac{y_0+y_1}{\sqrt{2}},\quad x_n:=\frac{y_0-y_1}{\sqrt{2}},
\quad x_i:=y_{i+1} \qquad (i=1,\dots,n-1).
\end{equation}
Then the properties (a1) and (a2) for this new coordinate system
$(x_0,\dots,x_n)$ are obvious.
Since the property that all of the Christoffel symbols
vanish along $\sigma$ is preserved under linear coordinate changes,
(a3) is also obtained.
\end{proof}

\section{Computations in $\R^{n+1}_1$}\label{app:Rn1}

We denote by the dot \lq$\cdot$\rq\ the canonical 
Lorentzian inner product of $\R^{n+1}_1$ with
signature $(-+\cdots+)$.
In this appendix, we compute $B:=B_F$ and $A:=A_F$
with respect to the canonical coordinate system
$(x_0,x_1,\dots,x_n)$ of $\R^{n+1}_1$.
Let $f(x_1,\dots,x_n)$ be a $C^2$-function of
$n$ variables defined on a neighborhood of 
the origin $o\in \R^n$. 
We set
$$
F=(f(x_1,\dots,x_n),x_1,\dots,x_n)
$$
and
$$
\vect{e}_0:=(1,0,\dots,0),\quad
\vect{e}_1:=(0,1,\dots,0),\quad \dots,\quad
\vect{e}_n:=(0,0,\dots,1).
$$
They give a canonical frame in $\R^{n+1}_1$
satisfying $\vect{e}_0\cdot \vect{e}_0=-1$.
Then $(x_1,\dots,x_n)$ gives a local coordinate system
of the domain of $F$, and we have
$$
F_{x_i}=f_{x_i}\vect{e}_0+\vect{e}_i
\qquad (i=1,\dots,n).
$$
We set
$$
S:=(s_{i,j}^{})_{i,j=1,\dots,n},\qquad
s_{i,j}^{}:=F_{x_i}\cdot F_{x_j} 
$$
for $i,j=1,\dots,n$,
where the dot \lq $\cdot$ \rq\ is 
the canonical inner product of
$\R^{n+1}_1$. Then we have
\begin{equation}
\label{eq:b1}
s_{i,j}^{}=\delta_{i,j}-f_{x_i}f_{x_j}\qquad \qquad (i,j=1,\dots,n),
\end{equation}
and
$$
S=I_n-(\nabla f)^T(\nabla f),
\qquad  \nabla f:=(f_{x_1},\dots,f_{x_n})
$$
hold, where $I_n$ is the identity matrix
and $(\nabla f)^T$ is the transpose of $(\nabla f)$.
Then we have
$$
B=\op{det}(S)=(1-\lambda_1)\cdots (1-\lambda_n),
$$
where $\lambda_1,\dots,\lambda_n$ are
eigenvalues of the matrix $(\nabla f)^T(\nabla f)$.
Since $(\nabla f)^T(\nabla f)$ is of rank $1$,
we may assume that $\lambda_2=\cdots=\lambda_n=0$.
Then we have
\begin{equation}
\label{eq:b2}
B=1-\lambda_1=1-\op{trace}((\nabla f)^T(\nabla f))
=1-(f_{x_1})^2-\cdots -(f_{x_n})^2.
\end{equation}
Using this, it can be easily checked that
the inverse matrix of $S$ is given by
$$
S^{-1}=I_n+\frac{1}{B}(\nabla f)^T(\nabla f).
$$
In particular, the cofactor matrix 
$\tilde S=(\tilde s^{i,j})_{i,j=1,\dots,n}$
of $S$ satisfies
\begin{equation}
\label{eq:b30}
\tilde S=B I_n+(\nabla f)^T(\nabla f),
\end{equation}
that is
\begin{equation}
\label{eq:b3}
\tilde s^{i,j}=B\delta_{i,j}+f_{x_i}f_{x_j}\qquad \qquad (i,j=1,\dots,n),
\end{equation}
where $\delta_{i,j}$ denotes Kronecker's delta.

On the other hand,
\begin{equation}
\label{eq:b4}
\tilde \nu=-(1,f_{x_1},\dots,f_{x_n})
\end{equation}
gives a normal vector field 
defined by \eqref{eq:nu0}.
Then the coefficients $\tilde h_{i,j}$
of the normalized second fundamental
form given in 
\eqref{eq:AF}
are written as
\begin{equation}
\label{eq:b5}
\tilde h_{i,j}=F_{x_i,x_j}\cdot \tilde \nu=f_{x_i,x_j}.
\end{equation}
Thus the matrix $\tilde h:=(\tilde h_{i,j})_{i,j=1,\dots,n}$ is just the
Hessian matrix of $f$.
By using the identity \eqref{eq:b2}, the function $A$ 
given in \eqref{eq:AF} can be computed as follows:
\begin{align*}
A&=\op{trace}(\tilde S \tilde h)=
\sum_{i,j=1}^n 
(B\delta_{i,j}+f_{x_i}f_{x_j})f_{x_i,x_j}\\
&=B \sum_{i=1}^n
\left(f_{x_i,x_i}
-\frac12 B_{x_i}f_{x_i}\right)
=B \triangle f -\frac 12 \nabla B\star \nabla f,
\end{align*} 
where \lq$\star$\rq\ is the canonical inner 
product of $\R^n$.

\begin{acknowledgement}
The authors thank Toshizumi Fukui
for fruitful discussions and the referees
for valuable comments.
\end{acknowledgement}

\end{document}